\documentclass[10pt]{amsart}

\usepackage{comment, amsmath, amssymb, amsgen, amsthm, amscd, xspace, color, epsfig, float, epic}
\usepackage[hypertex]{hyperref}

\makeatletter
\newcommand{\rmnum}[1]{\romannumeral #1}
\newcommand{\Rmnum}[1]{\expandafter\@slowromancap\romannumeral #1@}

\newcommand{\eps}{\epsilon}

\newcommand{\pt}{\partial}
\newcommand{\vf}{\varphi}
\newcommand{\ve}{\varepsilon}

\newcommand{\ad}{\operatorname{ad}}

\newcommand{\Hom}{\operatorname{Hom}}

\newcommand{\htt}{\mathrm{ht}}

\newcommand{\Ree}{\mathrm{Re}}
\newcommand{\Imm}{\mathrm{Im}}

\newcommand{\ra}{\rightarrow}

\newcommand{\pf}{\begin{proof}}
\newcommand{\epf}{\end{proof}}
\newcommand{\eq}{\begin{equation}}
\newcommand{\eeq}{\end{equation}}
\newcommand{\eqn}{\begin{equation*}}
\newcommand{\eeqn}{\end{equation*}}

\newcommand{\frb}{\mathfrak{b}}

\newcommand{\frg}{\mathfrak{g}}
\newcommand{\frh}{\mathfrak{h}}

\newcommand{\frn}{\mathfrak{n}}

\newcommand{\frsl}{\mathfrak{sl}}

\newcommand{\frsp}{\mathfrak{sp}}

\newcommand{\bbC}{\mathbb{C}}
\newcommand{\bbD}{\mathbb{D}}

\newcommand{\bbN}{\mathbb{N}}

\newcommand{\bbR}{\mathbb{R}}

\newcommand{\bbZ}{\mathbb{Z}}
\newcommand{\caA}{\mathcal{A}}

\newcommand{\caO}{\mathcal{O}}

\newtheorem{theorem}[equation]{Theorem}

\newtheorem{prop}[equation]{Proposition}
\newtheorem{lemma}[equation]{Lemma}

\theoremstyle{remark}
\newtheorem{remark}[equation]{Remark}

\theoremstyle{definition}
\newtheorem{definition}[equation]{Definition}

\newtheorem{example}[equation]{Example}

\numberwithin{equation}{section} 
\begin{document}

\title[Differential operators and Singular vectors]{Differential-operator representations of Weyl group and singular vectors in Verma modules}
\author{Wei Xiao}
\thanks{This work is supported by NSFC Grant No. 11326059.}
\address{College of Mathematics and Statistics, Shenzhen University,
Shenzhen, 518060, Guangdong, China}
\email{xiaow@szu.edu.cn}




\bigskip

\begin{abstract}
Given a suitable ordering of the positive root system associated with a semisimple Lie algebra, there exists a natural correspondence between Verma modules and related polynomial algebras. With this, the Lie algebra action on a Verma module can be interpreted as a differential operator action on polynomials, and thus on the corresponding truncated formal power series. We prove that the space of truncated formal power series is a differential-operator representation of the Weyl group $W$. We also introduce a system of partial differential equations to investigate singular vectors in the Verma module. It is shown that the solution space of the system in the space of truncated formal power series is the span of $\{w(1)\ |\ w\in W\}$. Those $w(1)$ that are polynomials correspond to singular vectors in the Verma module. This elementary approach by partial differential equations also gives a new proof of the well-known BGG-Verma Theorem.
\end{abstract}

\maketitle

%
%
\section{Introduction}
%
%
The most elementary infinite dimensional modules in the category $\caO$ introduced by Bernstein, Gelfand and Gelfand (abbreviated as BGG in what follows) are Verma modules \cite{BGG2}. Given a Borel subalgebra $\frb$ and a Cartan subalgebra $\frh\subset\frb$ of the semisimple Lie algebra $\frg$, the Verma module of weight $\lambda\in\frh^*$ is the induced module $M_\lambda:=U(\frg)\otimes_{U(\frb)}\bbC v_\lambda$, where $v_\lambda$ is a highest weight vector of weight $\lambda$. Research into Verma modules was originated in Verma's 1966 thesis \cite{V}, in which he gave a sufficient condition for the existence of nontrivial Hom spaces between Verma modules. The necessity of this condition was proved by BGG, who also obtained a new argument of the sufficiency \cite{BGG1}. Another approach to the necessity by Jantzen filtration and contravariant forms is also of independent interest \cite{J}.

One basic problem in this direction is to explicitly construct such a homomorphism if it exists. Any homomorphism between Verma modules is precisely determined by a weight vector called singular vector. Such a vector in a Verma module is characterized by the condition that it can be annihilated by the nilpotent radical $\frn$ of $\frb$ so that $\frb=\frh\oplus\frn$. In \cite{S}, Shapovalov introduced elements in $U(\frg)$ to find the determinant of his contravariant form. The Shapovalov elements can also be applied to generate singular vectors in Verma modules, although it seems laborious to explicitly give these elements in practice. Following Shapovalov's work, Lutsyuk derived a recurrence relation for obtaining singular vectors \cite{L}. After that, Malikov, Feigin and Fuks \cite{MFF} found formulas of singular vectors for $\frsl(n,\bbC)$ by considering products of complex powers of negative simple root vectors. But in general, it is difficult to verify that these products are well-defined.

Xu \cite{Xu1} overcame this difficulty for $\frsl(n,\bbC)$ by investigating the relation between singular vectors and partial differential equations. With an identification between a Verma module and its corresponding polynomial algebra, the action of $\frsl(n,\bbC)$ on the Verma module becomes a differential-operator action on the polynomial algebra. Xu constructed a system of second-order linear partial differential equations and pointed out that any singular vector in the Verma module coincides with a polynomial solution of the system. In order to solve the system, he considered the differential-operator action of $\frsl(n, \bbC)$ on a larger space of certain truncated formal power series, on which the arbitrary complex powers of negative simple root vectors are well-defined. It turns out that the space of truncated formal power series is a differential-operator representation of the symmetric group $S_n$, and the solution space of the differential system is the span of $\{\sigma(1)\ |\ \sigma\in S_n\}$.

It is reasonable to ask whether this is true in general, that is, when $\frsl(n,\bbC)$ is replaced by any other semisimple Lie algebra $\frg$. Xu himself treated the case of $\frsp(2n)$ in \cite{Xu2}. The complex powers of negative simple root vectors are well-defined and a system of partial differential equations can also be constructed to find singular vectors. However, the system seems to be quite difficult to solve even for small $n$. Only the solutions associated with $n=2$ have been fully determined. The arguments in \cite{Xu1} which depend on the ``good" properties of $\frsl(n,\bbC)$ does not work in general, we need to find another approach.

Let $\Phi\supset\Phi^+\supset\Delta$ be the root system of $(\frg,\frh)$ with the positive root system $\Phi^+$ corresponding to $\frb$ and the subset $\Delta$ of simple roots. The Weyl group $W$ is generated by all reflections $s_\beta$ with $\beta\in\Phi$. The dot action of $W$ on $\frh^*$ is defined by $w\cdot\lambda:=w(\lambda+\rho)-\rho$ for $w\in W$ and $\lambda\in\frh^*$. Here $\rho$ is the half sum of positive roots. Denote by $\bar\frn$ the dual space of $\frn$, such that $\frg=\bar\frn\oplus\frb$. Let $\{E_\beta,\beta\in\Phi;\ H_\alpha,\alpha\in\Delta\}$ be a Chevalley basis of $\frg$ (see \cite{H1}, Theorem 25.2). Then any ordering $\beta_1>\beta_2>\ldots>\beta_m$ imposed on $\Phi^+$ determines a PBW basis
\[
\left\{E^a:=E_{-\beta_1}^{a_1}E_{-\beta_2}^{a_2}\ldots E_{-\beta_m}^{a_m}\ \left|\  a=\sum_{i=1}^m a_i\eps_i\in\Gamma\right.\right\}
\]
of $U(\bar\frn)$, where $m:=|\Phi^+|$ and $\Gamma:=\sum_{i=1}^m\bbN\eps_i$ is the rank $m$ torsion-free additive semigroup with base elements $\eps_i$. The corresponding polynomial algebra is defined to be
\[
\caA:=\bbC[x_i\ |\ 1\leq i\leq m],
\]
with a basis
\[
\{x^a:=\prod_{i=1}^mx_i^{a_i}\ |\ a\in\Gamma \}.
\]
Then we have a linear isomorphism $\tau: M_\lambda\ra\caA$ given by
\[
\tau(E^av_\lambda)=x^a
\]
for $a\in\Gamma$. The polynomial algebra $\caA$ becomes a $U(\frg)$-module with the action
\[
u(f)=\tau(u(\tau^{-1}(f)))
\]
for $u\in U(\frg)$ and $f\in\caA$. This gives a differential-operator action of $U(\frg)$ on $\caA$ (Proposition \ref{prop1}). For convenience, we denote
\[
\zeta_\alpha:=H_{\alpha}|_\caA,\ \eta_\beta:=E_{-\beta}|_\caA\ \mbox{and}\ d_\beta:=E_{\beta}|_\caA,
\]
for $\alpha\in\Delta, \beta\in\Phi^+$. Thus any singular vector in $M_\lambda$ corresponds to a polynomial solution of the system of partial differential equations
\begin{equation}\label{int1}
d_\alpha(f)=0,
\end{equation}
for all $\alpha\in\Delta$ and unknown function $f$ in $\{x_i\ |\ i=1,2,\ldots,m\}$. We say that $f$ is {\it weighted}, with weight $\mu$, if there exists $\mu\in\frh^*$ such that $\zeta_\alpha(f)=\mu(H_{\alpha})f$ for all $\alpha\in\Delta$.

To make notation simpler, we often write $x_{\beta_i}=x_{i}$ for $i=1,2,\ldots,m$. Let $\alpha_1, \alpha_2,\ldots, \alpha_n$ be all the simple roots in $\Delta$. Denote
\[
\caA_0:=\bbC\left[x_{\beta}\ |\ \beta\in\Phi^+\backslash \Delta\right]
\]
and
\[
x^{\vec{z}}:=\prod_{i=1}^{n}x_{\alpha_i}^{z_i}\quad\mbox{for}\ \vec{z}=(z_1,z_2,\ldots,z_{n})\in\bbC^{n}.
\]
The space of truncated-up formal power series in $\{x_{\alpha_1},x_{\alpha_2}\ldots,x_{\alpha_n}\}$ over $\caA_0$ is
\[
\caA_1:=\left\{f=\sum_{\vec{j}\in\bbN^{n}}\left.\sum_{i=0}^pf_{\vec{z}^{i}-\vec{j}}x^{\vec{z}^i-\vec{j}}\ \right|\ p\in\bbN, \vec{z}^i\in\bbC^{n}, f_{\vec{z}^i-\vec{j}}\in\caA_0\right\}.
\]
Then $\caA_1$ contains $\caA$ and is also invariant under the action of $\{\zeta_\alpha, d_\beta, \eta_\beta\ |\ \alpha\in\Delta, \beta\in\Phi^+\}$. It is shown that for a suitable ordering of $\Phi^+$, we can define complex powers of $\eta_\alpha$ for all $\alpha\in\Delta$. Based on an effective use of the root posets, we first prove that there exists at most one solution (up to scalars) in $\caA_1$ for each weight (Lemma \ref{equni1}). This is enough to imply a differential-operator representation of the Weyl group $W$ on $\caA_1$.

\smallskip
\noindent{\bf Theorem A.} (Theorem \ref{thm1}) {\it
The space $\caA_1$ of truncated-up formal power series is a representation of the Weyl group $W$.}
\smallskip

At the same time, we deduce that $\{w(1)\ |\ w\in W\}\in\caA_1$ is a set of weighted solutions of the system \ref{int1}. Moreover, they do exhaust all the possibilities up to scalars. Hence we solve the system \ref{int1}.

\smallskip
\noindent{\bf Theorem B.} (Theorem \ref{thm2}) {\it
The solution space of the system $\ref{int1}$ in $\caA_1$ is spanned by $\{w(1)\ |\ w\in W\}$.}
\smallskip

In order to find all the polynomial solutions, a useful notation is needed. Given $\lambda,\mu\in\frh^*$, write $\mu\uparrow\lambda$ if there exists a positive root $\gamma\in\Phi^+$ such that $\mu=s_\gamma\cdot\lambda=\lambda-\langle\lambda+\rho,\gamma^\vee\rangle\gamma$ and
$\langle\lambda+\rho,\gamma^\vee\rangle\in\mathbb{Z}^{>0}$. We say that $\mu$ is {\it strongly linked} to $\lambda$ if $\mu=\lambda$ or there exist $\gamma_1,
\ldots, \gamma_r\in\Phi^+$ such that
\[
\mu=(s_{\gamma_1}\ldots
s_{\gamma_r})\cdot\lambda\uparrow(s_{\gamma_2}\ldots
s_{\gamma_r})\cdot\lambda\uparrow\ldots\uparrow
s_{\gamma_r}\cdot\lambda\uparrow\lambda.
\]

\smallskip
\noindent{\bf Theorem C.} (Theorem \ref{thm3}) {\it
Let $f\in\caA_1$ be a nonzero weighted solution of the system $\ref{int1}$, with weight $\mu$. Then $f$ is a polynomial if and only if $\mu$ is strongly linked to $\lambda$. In this case, $\tau^{-1}(f)$ is a singular vector of $M_\lambda$.}
\smallskip

The above results involve a new approach of the well known BGG-Verma Theorem (see the proof of Theorem \ref{BGG-V thm}). With our main theorems in hand, it is possible to explicitly write down the formulas of singular vectors in a proper PBW basis. This has been done for $\frsl(n,\bbC)$ in \cite{Xi2}. The formulas coincide with those in \cite{MFF} and \cite{DF} which come from different approaches. In the present paper, we also give some singular vectors for $\frsp(2n)$. More formulas of singular vectors by this approach will appear in future work. It was showed in \cite{Xi1} (i.e., Example 4.18) that many singular vectors in generalized Verma modules can be constructed from singular vectors in Verma modules. The singular vectors in generalized Verma modules determine homomorphisms between generalized Verma modules and invariant differential operators between homogeneous vector bundles. The approach in this paper would shed some light on these classical open-ended problems in representation theory and parabolic geometry.

At last, we briefly describe the contents of this paper. In Section 2, we recall the basic notions of Verma modules and derive a system of partial differential equations. In section 3, we get a differential-operator representation of the Weyl group on the space of truncated formal power series by an important lemma (Lemma \ref{eqgra1}). In section 4, we solve the system and give a new proof of the BGG-Verma Theorem. In section 5, we give a formula of singular vectors for $\frsp(2n)$. In section 6, we prove Lemma \ref{eqgra1} by investigating the Hasse diagram of root posets.

%
%
\section{Verma modules and differential equations}
%
%
In this section, we derive a system of partial differential equations to determine singular vectors in any given Verma module.

\subsection{Verma modules and Chavelley basis}
We now describe basic notions of Verma modules, referring to Humphreys \cite{H2} for full details. Let $\frg$ be a semisimple Lie algebra over $\bbC$, containing a fixed Cartan subalgebra $\frh$. Let $\Phi\subset\frh^*$ be the root system of $(\frg,\frh)$, with a fixed positive system $\Phi^+$ and the related simple system $\Delta\subset\Phi^+$. Denote by $\frg_\alpha$ the root subspace associated with $\alpha\in\Phi$. Then we have a Cartan decomposition $\frg=\bar\frn\oplus\frh\oplus\frn$, where $\bar\frn:=\bigoplus_{\alpha\in\Phi^+}\frg_{-\alpha}$ and $\frn:=\bigoplus_{\alpha\in\Phi^+}\frg_\alpha$. The corresponding Borel subalgebra is $\frb:=\frh\oplus\frn$.

Let $W$ be the associated Weyl group. Define the {\it dot action} of $w\in W$ on $\lambda\in\frh^*$ by $w\cdot\lambda=w(\lambda+\rho)-\rho$, where $\rho=\frac{1}{2}\sum_{\alpha\in\Phi^+}{\alpha}$. Denote by $\langle,\rangle$ the usual bilinear form on $\frh^*$ and by
$\alpha^\vee=2\alpha/\langle\alpha,\alpha\rangle$ the coroot of $\alpha\in\Phi$.

The {\it Verma module} of highest weight
$\lambda$ is defined by
\begin{equation*}
M_\lambda:=U(\frg)\otimes_{U(\frb)}\bbC v_\lambda,
\end{equation*}
where $v_\lambda$ is a highest weight vector of weight $\lambda\in\frh^*$. The PBW Theorem allows us to write
\[
M_\lambda\simeq U(\bar\frn)v_\lambda
\]
as a left $U(\bar\frn)$-module.

\begin{theorem}[\cite{H1}, Theorem 25.2] \label{Chavelley basis}
We can choose root vectors $E_\alpha\in\frg_\alpha$ and
$H_\alpha\in\frh$ such that, for all $\alpha, \beta\in\Phi$,
\begin{align*}
&[H_\alpha, E_\beta]=\langle\beta,\alpha^\vee\rangle E_\beta,\\
&[E_\alpha, E_{-\alpha}]=H_\alpha,\\
&[E_\alpha,
E_\beta]=N_{\alpha,\beta}E_{\alpha+\beta}\quad\quad\alpha+\beta\neq0,
\end{align*}
where constants $N_{\alpha, \beta}$ satisfy $N_{\alpha, \beta}=-N_{\beta, \alpha}.$ If $\alpha+\beta\not\in\Phi$, then $N_{\alpha, \beta}=0$. Otherwise
$N_{\alpha, \beta}=\pm(p+1)$, where $\beta+n\alpha$, with $-p\leq n\leq
q$, is the $\alpha$ string containing $\beta$.

\end{theorem}

\begin{lemma}\label{Chavelley basis2}
If all roots in $\Phi$ are of equal length, then for all $\alpha, \beta, \gamma\in\Phi$ such that $\alpha+\beta+\gamma=0$,
\[
N_{\alpha, \beta}=N_{\beta, \gamma}=N_{\gamma, \alpha}.
\]
\end{lemma}

\begin{proof}
Since all roots are of equal length, we have $\langle\alpha,\alpha\rangle=\langle\beta,\beta\rangle=\langle\gamma,\gamma\rangle$. For any $\delta\in\Phi$, it follows from Theorem \ref{Chavelley basis} that
\begin{align*}
[H_\alpha+H_\beta+H_\gamma, E_\delta]&=\left(\langle\delta,\alpha^\vee\rangle+\langle\delta,\beta^\vee\rangle+\langle\delta,\gamma^\vee\rangle\right) E_\delta\\
&=\frac{2\langle\delta,\alpha+\beta+\gamma\rangle}{\langle\alpha,\alpha\rangle} E_\delta=0.
\end{align*}
Therefore $H_\alpha+H_\beta+H_\gamma=0$. On the other hand, the Jocobi identity implies that
\[
[[E_\alpha, E_\beta], E_\gamma]+[[E_\beta, E_\gamma], E_\alpha]+[[E_\gamma, E_\alpha], E_\beta]=0.
\]
Thus
\[
N_{\alpha, \beta}[E_{-\gamma}, E_\gamma]+N_{\beta, \gamma}[E_{-\alpha}, E_\alpha]+N_{\gamma, \alpha}[E_{-\beta}, E_\beta]=0
\]
and $N_{\alpha, \beta}H_\gamma+N_{\beta, \gamma}H_\alpha+N_{\gamma, \alpha}H_\beta=0$. Since $H_\gamma=-H_\alpha-H_\beta$ and ${H_\alpha, H_\beta}$ are linearly independent, this can happen only when $N_{\alpha, \beta}=N_{\beta, \gamma}=N_{\gamma, \alpha}$.
\end{proof}

Denote $m:=|\Phi^+|$. Then any ordering $\beta_1>\beta_2>\ldots>\beta_m$ of $\Phi^+$ admits a natural bijection $$\iota:\Phi^+\ra \{1,2,\ldots,m\}$$ such that
$\iota(\beta_i)=i.$ Let
\[
\Gamma:=\sum_{i=1}^m\bbN\eps_i
\]
be the rank $m$ torsion-free additive semigroup with base elements $\eps_i$. Here $\bbN$ is the additive semigroup of nonnegative integers. For any
$a=\sum_{i=1}^ma_i\eps_i\in\Gamma,$
define
\begin{equation}\label{eq PBW basis1}
E^a:=E_{-\beta_1}^{a_1}E_{-\beta_2}^{a_m}\ldots E_{-\beta_m}^{a_m}\in U(\bar\frn).
\end{equation}
The set of all monomials $E^a$ ($a\in\Gamma$) is exactly a PBW basis of $U(\bar\frn)$. For simplicity, we sometimes use notations $a_\beta:=a_{\iota(\beta)}$ and $\eps_\beta:=\eps_{\iota(\beta)}$ for $a\in\Gamma$ and $\beta\in\Phi^+$. Then
\[
\Gamma_s:=\{a\in\Gamma\ |\ a_\alpha=0\ \mbox{for all}\ \alpha\in\Delta\}
\]
is a subsemigroup of $\Gamma$.

\begin{lemma}\label{commutator}
If $\alpha,\beta\in\Phi$, then for $k\in\bbN$,
\begin{equation*}
\begin{aligned}
E_\alpha^kE_\beta=&E_\beta E_\alpha^k+\binom{k}{1}[E_\alpha,E_\beta]E_\alpha^{k-1}+\binom{k}{2}[E_\alpha,[E_\alpha,E_\beta]]E_\alpha^{k-2}\\
&+\binom{k}{3}[E_\alpha,[E_\alpha,[E_\alpha,E_\beta]]]E_\alpha^{k-3},
\end{aligned}
\end{equation*}
where $\binom{k}{i}$ are the binomial coefficients.
\end{lemma}

\begin{proof}
It follows from Lemma 2.6 in \cite{H2} that
\[
E_\alpha^kE_\beta=\sum_{i=0}^k\binom{k}{i}(\ad E_\alpha)^{i}(E_\beta) E_\alpha^{k-i}.
\]
On the other hand, the $\alpha$ string containing $\beta$ contains at most four elements. Therefore
\[
(\ad E_\alpha)^{i}(E_\beta)=0
\]
unless $i\leq3$.
\end{proof}

\subsection{Differential operators}

Consider the polynomial algebra
\[
\caA=\bbC[x_i\ |\ 1\leq i\leq m]
\]
in $m$ variables. Then
\[
\{x^a:=\prod_{i=1}^mx_i^{a_i}\ |\ a\in\Gamma\}
\]
is a basis of $\caA$. The degree of $x^a$ is
\[
|a|:=\sum_{i=1}^ma_i.
\]
There exists a linear isomorphism $\tau: M_\lambda\ra\caA$ such that
\[
\tau(E^av_\lambda)=x^a
\]
for $a\in\Gamma$. Given $u\in U(\frg)$ and $f\in\caA$, we define
\begin{equation}\label{eqenv1}
u(f)=\tau(u(\tau^{-1}(f))).
\end{equation}
Thus $\caA$ is naturally a $U(\frg)$-module. Set $\pt_i=\pt/\pt x_i$. Similarly, we often write $x_\beta:=x_{\iota(\beta)}$ and $\pt_\beta:=\pt_{\iota(\beta)}$.
Denote by $A$ the Weyl algebra generated by $x_\beta$ and $\pt_\beta$ for $\beta\in\Phi^+$. A basis of $A$ is
\[
\{x^a\pt^b\ |\ a, b\in\Gamma;\ \pt^b:=\prod_{i=1}^m\pt_i^{b_i}\}.
\]
For $i\in\bbZ$, denote
\[
A^i:=\mathrm{span}\{x^a\pt^b\ |\ a, b\in\Gamma, |a|-|b|\leq i\}.
\]
Then $A^iA^j=A^{i+j}$ for $i, j\in\bbZ$ and we obtain a filtration
\[
\ldots \subset A^{-2}\subset A^{-1}\subset A^0 \subset A^1 \subset A^2 \subset \ldots.
\]
In particular, $\pt_k\in A^{-1}$ and $x_k\in A^1$ for $k=1, \ldots, m$. Similarly, we can define
\[
A_s^i:=\mathrm{span}\{x^a\pt^b\ |\ a\in\Gamma_s, b\in\Gamma, |a|-|b|\leq i\}\subset A^i.
\]
In this section we will be able to show that the action of $U(\frg)$ on $\caA$ are differential-operator action, that is, we have $u|_\caA\in A$ for any $u\in U(\frg)$ (Proposition \ref{prop1}). The following lemmas will be needed in the proof.

\begin{lemma}\label{basic lem2}
Given $\beta\in\Phi^+$, let $\eta(\beta,i)$ be the operator on $\caA$ such that
\[
\eta(\beta,i)x^a=\tau(E_{-\beta_1}^{a_1}\ldots E_{-\beta_{i}}^{a_{i}}E_{-\beta}E_{-\beta_{i+1}}^{a_{i+1}}\ldots E_{-\beta_m}^{a_m}v_\lambda),
\]
where $i=0,1,\ldots,m$ and $a\in\Gamma$. Then
\begin{equation*}
\eta(\beta,i)-x_\beta\in A_s^0.
\end{equation*}
\end{lemma}
\begin{proof}
Let $\htt\beta$ be the height of the root $\beta\in\Phi$ (see \cite{H1}). We use downward induction on $\htt\beta$ to prove the lemma. Start with $\beta\in\Phi^+$ which has the largest possible height. Then $[E_{-\beta}, E_{-\beta_j}]=0$ for $j=1,2,\ldots,m$. This yields
\[
\eta(\beta,i)x^a=\tau(E^{a+\epsilon_{\beta}}v_\lambda)=x^{a+\epsilon_{\beta}}=x_{\beta} x^a
\]
for all $a\in\Gamma$ and thus $\eta(\beta,i)-x_{\beta}=0\in A_s^0$ for $i=0,1,\ldots,m$. If $\htt\beta$ is not the largest, assume that $\eta(\beta', i)-x_{\beta'}\in A_s^0$ for all $\beta'\in\Phi^+$ such that $\htt\beta'>\htt\beta$ and $i=0,1,\ldots,m$. It follows from Lemma \ref{commutator} that
\[
E_{-\beta_{i}}^{a_{i}}E_{-\beta}=\sum_{j=0}^{3}\binom{a_i}{j}(\ad E_{-\beta_i})^{j}(E_{-\beta}) E_{-\beta_i}^{a_i-j}.
\]
Therefore we have
\begin{equation*}
\begin{aligned}
&\eta(\beta, i)x^a\\
=&\eta(\beta, i-1)x^a+c_{1}a_i\eta(\beta+\beta_i, i-1)x^{a-\ve_i}+c_{2}a_i(a_i-1)\eta(\beta+2\beta_i, i-1)x^{a-2\ve_i}\\
&+c_{3}a_i(a_i-1)(a_i-2)\eta(\beta+3\beta_i, i-1)x^{a-3\ve_i}\\
=&\eta(\beta, i-1)x^a+c_{1}\eta(\beta+\beta_i, i-1)\pt_ix^a+c_{2}\eta(\beta+2\beta_i, i-1)\pt_i^2x^a\\
&+c_{3}\eta(\beta+3\beta_i, i-1)\pt_i^3x^a,\\
\end{aligned}
\end{equation*}
with constants $c_{1}$, $c_2$ and $c_3$. Here $\eta(\beta+k\beta_i, i-1)=0$ if $\beta+k\beta_i\not\in\Phi$ for $k=1,2$ or $3$. So we have
\begin{equation}\label{eta1}
\begin{aligned}
\eta(\beta, i)=&\eta(\beta, i-1)+c_{1}\eta(\beta+\beta_i, i-1)\pt_i\\
&+c_{2}\eta(\beta+2\beta_i, i-1)\pt_i^2+c_{3}\eta(\beta+3\beta_i, i-1)\pt_i^3.
\end{aligned}
\end{equation}
Since $\htt(\beta+k\beta_i)>\htt\beta$, the induction hypothesis can be applied, yielding $\eta(\beta, i)-\eta(\beta, i-1)\in A_s^0$ for $i=1, \ldots, m$. It is obvious that $\eta(\beta, i)=x_\beta$ when $i=\iota(\beta)$ or $\iota(\beta)-1$ (that is, $\beta=\beta_i$ or $\beta_{i+1}$). Hence $\eta(\beta, i)-x_\beta\in A_s^0$ for $i=0, 1, \ldots, m$.
\end{proof}

\begin{lemma}\label{basic lem3}
Given $\beta\in\Delta$, let $d(\beta,i)$ be the operator on $\caA$ such that
\[
d(\beta,i)x^a=\tau(E_{-\beta_1}^{a_1}\ldots E_{-\beta_{i}}^{a_{i}}E_{\beta}E_{-\beta_{i+1}}^{a_{i+1}}\ldots E_{-\beta_m}^{a_m}v_\lambda).
\]
where $i=0,1,\ldots,m$ and $a\in\Gamma$. Then
\[
d(\beta,i)-\sum_{j> i, \beta_j-\beta\in\Phi^+}N_{\beta,-\beta_j}x_{\beta_j-\beta}\pt_j
\in A^{-1}.
\]
\end{lemma}
\begin{proof}
Fix $\beta\in\Delta$. We use downward induction on $i$ to prove the lemma.
The case $i=m$ is obvious ($d(\beta,m)=0$). Suppose that we already have the result for $d(\beta, i)$. Consider the operator $d(\beta,i-1)$. If $\beta=\beta_i$, then
\[
[E_\beta, E_{-\beta_i}^{a_i}]=a_iE_{-\beta_i}^{a_i-1}H_\beta-a_i(a_i-1)E_{-\beta_i}^{a_i-1}.
\]
Therefore
\[
\begin{aligned}
d(\beta, i-1)x^a-d(\beta, i)x^a=&a_i\left(\langle\lambda,\beta^\vee\rangle-a_j\sum_{j>i}\langle\beta_j,\beta^\vee\rangle\right)x^{a-\ve_i}-a_i(a_i-1)x^{a-\ve_i}\\
=&\left(\langle\lambda,\beta^\vee\rangle-\sum_{j>i}\langle\beta_j,\beta^\vee\rangle x_j\pt_j-x_i\pt_i\right)\pt_ix^a.
\end{aligned}
\]
So we obtain
\begin{equation}\label{d1}
d(\beta, i-1)-d(\beta, i)=\left(\langle\lambda,\beta^\vee\rangle-\sum_{j>i}\langle\beta_j,\beta^\vee\rangle x_j\pt_j-x_i\pt_i\right)\pt_i\in A^{-1}.
\end{equation}
If $\beta\neq\beta_i$, it follows from Lemma \ref{commutator} that
\begin{equation}\label{d2}
\begin{aligned}
d(\beta, i)=&d(\beta, i-1)+\tilde{c}_1\eta(\beta_{i}-\beta, i-1)\pt_i\\
&+\tilde{c}_{2}\eta(2\beta_i-\beta, i-1)\pt_i^2+\tilde{c}_{3}\eta(3\beta_i-\beta, i-1)\pt_i^3,
\end{aligned}
\end{equation}
with constants $\tilde{c}_1$, $\tilde{c}_2$ and $\tilde{c}_3$. Here $\eta(k\beta_i-\beta,i-1)=0$ if $k\beta_i-\beta\not\in\Phi$ for $k=1,2$ or $3$. In fact, if $k\beta_i-\beta\in\Phi$, then $k\beta_i-\beta\in\Phi^+$ since $\htt(k\beta_i-\beta)\geq k-1\geq 0$ for $\beta\in\Delta$. In particular, $\tilde{c}_1=N_{-\beta_{i},\beta}=-N_{\beta,-\beta_i}$ by Theorem \ref{Chavelley basis}. Using Lemma \ref{basic lem2}, together with (\ref{d1}), we obtain
\begin{equation*}
\left\{\begin{aligned}
&d(\beta, i-1)-d(\beta, i)\in A^{-1}\qquad\qquad\qquad\qquad\qquad\ \mbox{if}\ \beta_i-\beta\not\in\Phi^+;\\
&d(\beta, i-1)-d(\beta, i)-N_{\beta,-\beta_i}x_{\beta_i-\beta}\pt_i\in A^{-1}\quad\qquad \mbox{if}\ \beta_i-\beta\in\Phi^+.
\end{aligned}
\right.
\end{equation*}
By the induction hypothesis, we get the asserted result for $d(\beta, i-1)$.
\end{proof}

\begin{prop}\label{prop1}
If $u\in U(\frg)$, then $u|_\caA\in A$. In particular
\begin{itemize}
\item[$\mathrm{(\rmnum{1})}$] For $\alpha\in\Delta$,
\[
    \zeta_\alpha:=H_{\alpha}|_\caA=\langle\lambda,\alpha^\vee\rangle-\sum_{\beta\in\Phi^+}
    \langle\beta,\alpha^\vee\rangle x_\beta\pt_\beta.
    \]

\item[$\mathrm{(\rmnum{2})}$] For $\beta\in\Phi^+$,
    \[
    \eta_\beta:=E_{-\beta}|_\caA=x_\beta+\sum_{a\in\Gamma_s,b\in\Gamma, |a|\leq |b|}C_{a,b}^{\beta} x^a\pt^b,
    \]
    where $C_{a,b}^\beta$ are constants and nonzero for only finitely many pairs $(a,b)$.

\item[$\mathrm{(\rmnum{3})}$] For $\beta\in\Phi^+$, denote $D_\beta:=\sum_{\gamma,\gamma-\beta\in\Phi^+}N_{\beta,-\gamma}x_{\gamma-\beta}\pt_\gamma$. Then
\[
    d_\beta:=E_{\beta}|_\caA=D_\beta
    +\sum_{a,b\in\Gamma, |a|<|b|}\tilde{C}_{a,b}^\beta x^a\pt^b,
    \]
    where $\tilde{C}_{a,b}^\beta$ are constants and nonzero for only finitely many pairs $(a,b)$.
\end{itemize}
\end{prop}
\begin{proof}
(\rmnum{1}) Recall that
\[
H_\alpha(E^av_\lambda)=\left(\langle\lambda,\alpha^\vee\rangle-\sum_{\beta\in\Phi^+}\langle\beta,\alpha^\vee\rangle a_\beta\right)E^av_\lambda.
\]
The statement follows immediately.

(\rmnum{2}) Note that $E_{-\beta}|_\caA=\eta(\beta, 0)$ and thus the statement follows from  Lemma \ref{basic lem2}.

(\rmnum{3}) If $\beta\in\Delta$, the statement is a consequence of Lemma \ref{basic lem3} since $E_{\beta}|_\caA=d(\beta, 0)$. It remains to consider the general case, with $\beta\in\Phi^+$ arbitrary. Given $i,j\in\{1,2,\ldots,m\}$, if $\beta_i+\beta_j\in\Phi$, then
\[
\begin{aligned}
\left[D_{\beta_i},D_{\beta_j}\right]=&\left[\sum_{\gamma_1,\gamma_1-\beta_i\in\Phi^+}N_{\beta_i,-\gamma_1}x_{\gamma_1-\beta_i}\pt_{\gamma_1},
\sum_{\gamma_2,\gamma_2-\beta_j\in\Phi^+}N_{\beta_j,-\gamma_2}x_{\gamma_2-\beta_j}\pt_{\gamma_2}\right]\\
=&\sum_{\gamma_2,\gamma_1-\beta_i,\gamma_1=\gamma_2-\beta_j\in\Phi^+}
N_{\beta_i,-\gamma_1}N_{\beta_j,-\gamma_2}x_{\gamma_1-\beta_i}\pt_{\gamma_2}\\
&-\sum_{\gamma_1,\gamma_2-\beta_j,\gamma_2=\gamma_1-\beta_i\in\Phi^+}
N_{\beta_j,-\gamma_2}N_{\beta_i,-\gamma_1}x_{\gamma_2-\beta_j}\pt_{\gamma_1}\\
=&\sum_{\gamma,\gamma-\beta_i-\beta_j\in\Phi^+}
(N_{\beta_i,\beta_j-\gamma}N_{\beta_j,-\gamma}-N_{\beta_j,\beta_i-\gamma}N_{\beta_i,-\gamma})
x_{\gamma-\beta_i-\beta_j}\pt_{\gamma}\\
=&N_{\beta_i,\beta_j}\sum_{\gamma,\gamma-\beta_i-\beta_j\in\Phi^+}
N_{\beta_i+\beta_j,-\gamma}x_{\gamma-\beta_i-\beta_j}\pt_\gamma\\
=&N_{\beta_i,\beta_j}D_{\beta_i+\beta_j}.
\end{aligned}
\]
Here $N_{\beta_i,\beta_j-\gamma}=0$ if $\beta_j-\gamma\not\in\Phi$ and $N_{\beta_j,\beta_i-\gamma}=0$ if $\beta_i-\gamma\not\in\Phi$. The second last equation follows from the Jacobi identity. In view of the fact that
\[
[E_{\beta_i}, E_{\beta_j}]=N_{\beta_i,\beta_j}E_{(\beta_i+\beta_j)},
\]
we can prove the asserted result by induction on $\htt\beta$.
\end{proof}

Define the commutator of two differential operator $d$ and $\bar d$ by
\[
[d,\bar d]=d\bar d-\bar d d.
\]
Recall that the map $\iota:\Phi^+\ra \{1,2,\ldots,m\}$ determines an ordering of $\Phi^+$. It is easy to see that the formulas of $\zeta_\alpha$, $\eta_\beta$ and $d_\beta$ depend on $\iota$.

\begin{definition}\label{gd order}
We say that $\iota$ is a {\it good ordering} of $\Phi^+$ if $x_\alpha$ and $\eta_\alpha$ commute for all $\alpha\in\Delta$, that is,
\[
[x_\alpha,\eta_\alpha]=0.
\]
\end{definition}

Let $\alpha_1, \alpha_2, \ldots, \alpha_n$ be all the simple roots in $\Phi^+$ (i.e., $\Delta=\{\alpha_1, \alpha_2, \ldots, \alpha_n\}$). If $\beta, \gamma\in\Phi^+$, then $\beta-\gamma=\sum_{i=1}^nc_i\alpha_i$ with $c_i\in\bbZ$. We write $\beta>\gamma$ if there exists an index $k\in\{1,2,\ldots,n\}$ such that $c_1=c_2=\ldots=c_{k-1}=0$ and $c_k>0$. This defines a lexicographic ordering $\beta_1>\beta_2>\ldots>\beta_m$ on $\Phi^+$ (with $\alpha_1>\alpha_2>\ldots>\alpha_n$).

\begin{lemma}[The existence of good ordering]\label{gd order2}
The above lexicographic ordering is a good ordering of $\Phi^+$.
\end{lemma}

\begin{proof}
We still denote the above lexicographic ordering by $\iota$. Then we have $\iota(\beta)<\iota(\gamma)$ if and only if $\beta>\gamma$ for $\beta, \gamma\in\Phi^+$. Fix $\beta\in\Phi^+$ and $i\in\{0,1,\ldots,m\}$. We claim that $$[x_j,\eta(\beta,i)]=0$$ for $j\geq\max\{\iota(\beta), i+1\}$. The lemma is an immediate consequence of this claim. Indeed, set $i=0$ and $j=\iota(\beta)\geq\max\{\iota(\beta), 0+1\}$. Then $x_{j}=x_\beta$ commutes with $\eta(\beta,0)=\eta_\beta$.

As in Lemma \ref{basic lem2}, we use downward induction on $\htt\beta$ to prove the claim. If $\beta\in\Phi^+$ has the largest possible height, then $\eta(\beta,i)=x_{\beta}$ for $i=0,1,\ldots,m$ and there is nothing to prove. Now consider arbitrary $\beta\in\Phi^+$. For the induction step, suppose that the claim is true for any $\beta'\in\Phi^+$ with $\htt\beta'>\htt\beta$. If $i=\iota(\beta)$ or $\iota(\beta)-1$ (i.e., $\beta=\beta_i$ or $\beta_{i+1}$), then $\eta(\beta, i)=x_\beta$ and there is also nothing to prove. If $i=\iota(\beta)+1$. Recall the equation (\ref{eta1}):
\begin{equation*}
\begin{aligned}
\eta(\beta, i)=&\eta(\beta, i-1)+c_{1}\eta(\beta+\beta_i, i-1)\pt_i\\
&+c_{2}\eta(\beta+2\beta_i, i-1)\pt_i^2+c_{3}\eta(\beta+3\beta_i, i-1)\pt_i^3.
\end{aligned}
\end{equation*}
Assume that $j\geq\max\{\iota(\beta), i+1\}=i+1$. It follows immediately that $[
x_j,\pt_i]=0$. If $\beta+k\beta_i\in\Phi^+$ ($k=1,2$ or $3$), then $\beta+k\beta_i\geq\beta$ and $\iota(\beta+k\beta_i)\leq\iota(\beta)=i-1$. So $j\geq i+1>\max\{\iota(\beta+k\beta_i), (i-1)+1\}$. The induction hypothesis can be applied, yielding $[x_j, \eta(\beta+k\beta_i,i-1)]=0$ and thus $[x_j, \eta(\beta,i)]=0$. Similarly, we can prove the claim for $i=\iota(\beta)+2,\ldots, m$ by a subsidiary induction on $i$, starting with $i=\iota(\beta)+1$. With (\ref{eta1}) in hand, we can also prove the claim for $i=\iota(\beta)-2,\ldots,1,0$ by a subsidiary downward induction on $i$, starting with $i=\iota(\beta)-1$ (keeping in mind that $\iota(\beta)=\max\{\iota(\beta), i+1\}$ in this case). In fact, assume that $i\leq\iota(\beta)-1$. If $[x_{j}, \eta(\beta, i)]=0$ for $j\geq \iota(\beta)$, rewrite (\ref{eta1}) as
\begin{equation*}
\begin{aligned}
\eta(\beta, i-1)=&\eta(\beta, i)-c_{1}\eta(\beta+\beta_i, i-1)\pt_i\\
&-c_{2}\eta(\beta+2\beta_i, i-1)\pt_i^2-c_{3}\eta(\beta+3\beta_i, i-1)\pt_i^3.
\end{aligned}
\end{equation*}
Since $j\geq \iota(\beta)>i$, one has $[x_j,\pt_i]=0$ and $j\geq\max\{\iota(\beta+k\beta_i), (i-1)+1\}$ for $k=1, 2$ or $3$. The induction hypothesis yields $[x_j, \eta(\beta+k\beta_i,i-1)]=0$ and thus $[x_j, \eta(\beta,i-1)]=0$.
\end{proof}

\subsection{Differential equations and truncated-up formal power series}

From now on, assume that $\iota$ is always a good ordering. We say that a weight vector $v\in M_\lambda$ is a {\it singular vector} if $\frn\cdot v=0.$ As a consequence of the definition, we have:

\begin{prop}
A weight vector $v\in M_\lambda$ is a singular vector if and only if
\[
d_{\alpha}(\tau(v))=0
\]
for all $\alpha\in\Delta$.
\end{prop}

\begin{definition}\label{pde for singular}
We can define {\it the system of partial differential equations for singular vectors of $M_\lambda$} by
\begin{equation}\label{eqsin1}
d_\alpha(f)=0
\end{equation}
for all $\alpha\in\Delta$ and unknown function $f$ in $\{x_{\beta}\ |\ \beta\in\Phi^+\}$. If there is $\mu\in\frh^*$ such that $\zeta_\alpha(f)=\mu(H_\alpha)f=\langle\mu,\alpha^\vee\rangle f$ for all $\alpha\in\Delta$, we say that $f$ is {\it weighted}, with weight $\mu$.
\end{definition}

It follows immediately from the definition that the constant polynomial $1$ is a weighted solution, with weight $\lambda$.
If $f$ is a weighted polynomial solution of (\ref{eqsin1}), then $\tau^{-1}(f)$ is a singular vector of $M_\lambda$ and vice versa. In order to solve the system (\ref{eqsin1}), we need a proper space of functions. Define the polynomial algebra
\[
\caA_0:=\bbC\left[x_{\beta}\ |\ \beta\in\Phi^+\backslash \Delta\right]
\]
with basis
\[
\{x^a:=\prod_{i=1}^mx_i^{a_i}\ |\ a\in\Gamma_s\}.
\]
Define monomials
\[
x^{\vec{z}}:=\prod_{i=1}^{n}x_{\alpha_i}^{z_i}
\]
for $\vec{z}=(z_1,z_2,\ldots,z_{n})\in\bbC^{n}$. Let
\[
\caA_1:=\left\{f=\sum_{\vec{j}\in\bbN^{n}}\left.\sum_{i=0}^pf_{\vec{z}^i-\vec{j}}x^{\vec{z}^i-\vec{j}}\ \right|\ p\in\bbN, \vec{z}^i\in\bbC^{n}, f_{\vec{z}^i-\vec{j}}\in\caA_0\right\}
\]
be the space of truncated-up formal power series in $\{x_{\alpha_1},x_{\alpha_2}\ldots,x_{\alpha_n}\}$ over $\caA_0$. Evidently $\caA_1$ contains $\caA$ and is invariant under the action of $\{x_\beta,\pt_\beta\ |\ \beta\in\Phi^+\}$. Hence it is invariant under the action of $\{\zeta_\alpha,d_\beta,\eta_\beta\ |\ \alpha\in\Delta, \beta\in\Phi^+\}$ by Proposition \ref{prop1}.

By Proposition \ref{prop1} (\rmnum{2}), we can define differential operators
\begin{equation}\label{eqet2}
\begin{aligned}
\eta_\alpha^c=&\left(x_{\alpha}+\sum_{a\in\Gamma_s,b\in\Gamma, |a|\leq |b|}C_{a,b}^{\alpha} x^a\pt^b\right)^c\\
=&\sum_{p=0}^\infty\frac{\langle c\rangle_p}{p!}x_{\alpha}^{c-p}\left(\sum_{a\in\Gamma_s,b\in\Gamma, |a|\leq |b|}C_{a,b}^{\alpha} x^a\pt^b\right)^p
\end{aligned}
\end{equation}
on $\caA_1$ for $\alpha\in\Delta$ and $c\in\bbC$. Here $\langle c\rangle_p=c(c-1)\ldots(c-p+1)$. Since $x_{\alpha}$ and $\eta_\alpha-x_\alpha$ commute for any good ordering $\iota$, we get
\[
\eta_\alpha^{c_1}\eta_\alpha^{c_2}=\eta_\alpha^{c_1+c_2}\qquad\mbox{for}\ c_1,c_2\in\bbC.
\]

\begin{lemma}\label{eqcom1}
If $\alpha,\beta\in\Delta$ and $c\in \bbC$, then
\[
[d_{\beta},\eta_{\alpha}^c]=c\delta_{\alpha,\beta}\eta_{\alpha}^{c-1}(1-c+\zeta_{\alpha})\quad\mbox{and}\quad
[\zeta_{\beta},\eta_{\alpha}^{c}]=-c\langle\alpha,\beta^\vee\rangle\eta_{\alpha}^c.
\]
Here $\delta_{\alpha,\beta}=0$ unless $\alpha=\beta$. In particular, $\delta_{\alpha,\alpha}=1$.
\end{lemma}
\begin{proof}
Fix $\alpha\in\Delta$ and set $\Gamma_\alpha:=\left\{a\in\Gamma\ |\ a_\alpha=0\right\}\supset\Gamma_s$. Let $\bbD(\caA_1)$ be the algebra of differential operators on $\caA_1$. Denote
\[
S_{\alpha, c}:=\left\{D=\sum_{p\in\bbZ}x_\alpha^{c-p}D_p\in\bbD(\caA_1)\ \left|\ D_p=\sum_{a\in\Gamma_\alpha,b\in\Gamma}P_{p,a,b}(c)x^{a}\pt^b\in A\right.\right\}.
\]
For any $D\in S_{\alpha, c}$, the set $\{p<0\ |\ D_p\neq0\}$ is finite and $P_{p,a,b}(c)$ are polynomials in $c$. It is not difficult to verify that $S_{\alpha, c}$ is a bimodule of the Weyl algebra $A$.

It follows from (\ref{eqet2}) that $\eta_{\alpha}^{c}, \eta_{\alpha}^{c-1}\in S_{\alpha,c}$. Moreover, we have $d_{\beta}\eta_{\alpha}^c$, $\eta_{\alpha}^c d_{\beta}$, $\eta_{\alpha}^{c-1}(1-c+\zeta_{\alpha})\in S_{\alpha,c}$ since $d_\beta, \zeta_\alpha\in A$. Now we can assume that
\[
[d_{\beta},\eta_{\alpha}^c]-c\delta_{\alpha,\beta}\eta_{\alpha}^{c-1}(1-c+\zeta_{\alpha})=\sum_{p\in\bbZ} x_{\alpha}^{c-p}\sum_{a\in\Gamma_\alpha,b\in\Gamma}Q_{p,a,b}(c)x^{a}\pt^b,
\]
where $Q_{p,a,b}(c)$ are polynomials in $c$. Note that
\[
[E_{\beta}, E_{-\alpha}^k]=k\delta_{\alpha,\beta}E_{-\alpha}^{k-1}(1-k+H_{\alpha})\quad\mbox{for}\ k\in\bbN.
\]
Therefore $Q_{p,a,b}(k)=0$ for $k\in\bbN$. But a nonzero polynomial has only finitely many roots. So $Q_{p,a,b}(c)\equiv0$ and
\[
[d_{\beta},\eta_{\alpha}^c]=c\delta_{\alpha,\beta}\eta_{\alpha}^{c-1}(1-c+\zeta_{\alpha})\quad\mbox{for}\ c\in\bbC.
\]
In a similar spirit, the equation
\[
[H_{\beta}, E_{-\alpha}^k]=-k\langle\alpha,\beta^\vee\rangle E_{-\alpha}^k\quad\mbox{for}\ k\in\bbN
\]
implies
\[
[\zeta_{\beta},\eta_{\alpha}^{c}]=-c\langle\alpha,\beta^\vee\rangle\eta_{\alpha}^c\quad\mbox{for}\ c\in\bbC.
\]
\end{proof}

%
%
\section{Differential-operator representations of $W$}
%
%
In this section, we show that $\caA_1$ is a differential-operator representation of the Weyl group $W$.

\subsection{Degree and leading term}

\begin{lemma}\label{eqwt1}
Let $f\in\caA_1$ be a weighted function with weight $\mu=\lambda-\sum_{i=1}^nz_i\alpha_i$. Then $f$ can be written as
\[
f=\sum_{\vec{j}\in\bbN^{n}}f_{\vec{z}-\vec{j}}x^{\vec{z}-\vec{j}},
\]
where $\vec{z}=(z_1,z_2,\ldots,z_{n})$ and $f_{\vec{z}-\vec{j}}\in\caA_0$. Moreover, if $f_{\vec{z}-\vec{j}}\neq0$ for $\vec{j}=(j_1,j_2,\cdots,j_n)$, then
\[
\sum_{i=1}^nj_i-\deg(f_{\vec{z}^i-\vec{j}})\in\bbN.
\]
\end{lemma}
\begin{proof}
In view of Proposition \ref{prop1} (\rmnum{1}), we have $[\zeta_\alpha, x_\beta^c]=\langle-c\beta, \alpha^\vee\rangle x_\beta^c$ and $\zeta_\alpha(1)=\langle\lambda, \alpha^\vee\rangle$ for $\alpha\in\Delta$, $\beta\in\Phi^+$ and $c\in\bbC$. Let $x^ax^{\vec{z}'}$ be any nonzero term of $f$ (up to a coefficient) with $a\in\Gamma_s$ and $\vec{z}'=(z'_1,z'_2,\ldots,z'_{n})$. The weight of $x^ax^{\vec{z}'}$ is
\[
-\sum_{i=1}^m{a_i\beta_i}-\sum_{i=1}^n{z'_i}\alpha_i+\lambda=\mu=\lambda-\sum_{i=1}^nz_i\alpha_i.
\]
It follows that
\[
\sum_{i=1}^n{(z_i-z'_i)}\alpha_i=\sum_{i=1}^m{a_i\beta_i}.
\]
Then $\vec{z}-\vec{z}'\in\bbN^n$. Denote $\vec{j}=\vec{z}-\vec{z}'$. The first statement is evident.

If $f_{\vec{z}-\vec{j}}\neq0$, let $x^a(a\in\Gamma_s)$ be a nonzero term of $f_{\vec{z}-\vec{j}}$ (up to a coefficient) with the largest degree. Then $\deg(f_{\vec{z}-\vec{j}})=\deg(x^a)=\sum_{i=1}^ma_i$. Since $x^ax^{\vec{z}-\vec{j}}$ and $f$ have the same weights, we obtain
\[
-\sum_{i=1}^m{a_i\beta_i}-\sum_{i=1}^n(z_i-j_i)\alpha_i+\lambda=\mu=\lambda-\sum_{i=1}^nz_i\alpha_i.
\]
Considering the height of the above equation, we have
\[
\sum_{i=1}^nj_i=\sum_{i=1}^m{a_i\htt\beta_i}.
\]
Therefore,
\[
\sum_{i=1}^nj_i-\deg(f_{\vec{z}^i-\vec{j}})=\sum_{i=1}^m{a_i(\htt\beta-1)}\in\bbN.
\]
\end{proof}
We want to define the ``degree'' of a weighted $f\in\caA_1$ as in the case of polynomials. In the above setting, it is natural to define
\[
\deg(x^ax^{\vec{z}-\vec{j}})=\deg(x^a)+\sum_{i=1}^nz_i-\sum_{i=1}^nj_i.
\]
In a similar spirit, we can define {\it the degree} of a weighted $f\in\caA_1$ to be
\[
\deg(f):=\sum_{i=1}^nz_i-\min\left\{\left.\sum_{i=1}^nj_i-\deg(f_{\vec{z}^i-\vec{j}})\ \right|\ f_{\vec{z}^i-\vec{j}}\neq0\right\}.
\]
This is well-defined because of the above lemma. It is in a sense the highest degree of all the nonzero terms of $f$. Recall that we say a nonzero term of a polynomial $f$ is a {\it leading term} if has the largest possible degree $\deg(f)$. Similarly, we say a nonzero term of a weighted $f\in\caA_1$ is a leading term if its degree is $\deg(f)$. Denote by $T(f)$ the sum of all leading terms of $f$. The following lemma is an immediate consequence of Proposition \ref{prop1} and (\ref{eqet2}).

\begin{lemma}\label{eqwt2}
If $f\in\caA_1$ is weighted with weight $\mu$, then so is $T(f)$. Moreover,
\[
T(\eta_\alpha^c(f))=x_\alpha^c T(f)
\]
for $\alpha\in\Delta$ and $c\in\bbC$. If $D_\beta(T(f))\neq0$, then
\[
T(d_\beta(f))=D_\beta(T(f))
\]
for $\beta\in\Phi^+$.
\end{lemma}

Define a matrix  $A(\Phi)(x)=(a_{\beta,\gamma})_{m\times(m-n)}$, such that for $\beta\in\Phi^+,\gamma\in\Phi^+\backslash\Delta$,
\[
a_{\beta,\gamma}=\left\{\begin{aligned}
&N_{\beta,-\gamma}x_{\gamma-\beta}\qquad\mbox{if}\ \gamma-\beta\in\Phi^+\\
&0\ \qquad\qquad\qquad\mbox{otherwise}.
\end{aligned}
\right.
\]
Set $A(\Phi):=A(\Phi)(1,1,\cdots,1)$, that is, $x_\beta=1$ for all $\beta\in\Phi^+$. The following lemma will be useful.

\begin{lemma}\label{eqgra1}
The matrix $A(\Phi)$ has full rank $m-n$.
\end{lemma}

The proof of this lemma is different from our main ideas. We will show this in the last section.
\begin{lemma}\label{equni1}
Let $f\in\caA_1$ be a nonzero weighted solution of the system $\ref{eqsin1}$, with weight $\mu$. Then $f$ is unique $($up to a scalar$)$.
\end{lemma}

\begin{proof}
Evidently $d_\alpha(f)=0$ for $\alpha\in\Delta$ implies $d_\beta(f)=0$ for $\beta\in\Phi^+$. As a consequence of Lemma \ref{eqwt2}, we obtain $D_\beta (T(f))=0$. In other words,
\[
\sum_{\gamma,\gamma-\beta\in\Phi^+}N_{\beta,-\gamma}x_{\gamma-\beta}\pt_\gamma(T(f))=0.
\]
Set $X=(\pt_\gamma(T(f)))_{(m-n)\times1}$. We get a system of linear equations $A(\Phi)(x)X=\mathbf{0}$, with coefficient matrix $A(\Phi)(x)$ and unknowns $\pt_\gamma(T(f))$ for $\gamma\in\Phi^+\backslash\Delta$. With Lemma \ref{eqgra1} in hand, we claim that $X=\mathbf{0}$. Indeed, since $A(\Phi)$ has full rank, we can choose an $(m-n)\times(m-n)$ submatrix $M$ of $A(\Phi)(x)$ such that $\det(M)$ is nonzero when $x_\beta=1$ for all $\beta\in\Phi^+$. Then $MX=\mathbf{0}$ and $\det(M)$ is a nonzero polynomial in $\caA$. Let $M^*$ be the adjugate matrix of $M$. We obtain
\[
\mathbf{0}=M^*(MX)=\det(M)X
\]
and thus $X=\mathbf{0}$.

Note that $T(f)$ is also weighted, with weight $\mu$. Assume that $\lambda-\mu=\sum_{i=1}^nz_i\alpha_i$ for $z_i\in\bbC$. Since $\pt_\gamma(T(f))=0$ for $\gamma\in\Phi^+\backslash\Delta$, we must have $T(f)=cx^{\vec{z}}$ by Lemma \ref{eqwt1}, where $c$ is a nonzero constant and $\vec{z}=(z_1,z_2,\ldots,z_{n})$. Let $f'$ be another weighted solution of the system \ref{eqsin1}, with weight $\mu$. Then there exists a nonzero constant $c'$ such that $T(f')=c'x^{\vec{z}}$. So $c'f-cf'$ is also a weighted solution with weight $\mu$. If $c'f-cf'\neq0$, there exists a nonzero constant $c''$ such that $T(c'f-cf')=c''x^{\vec{z}}$ and thus $c''=c'c-cc'=0$, a contradiction. Hence $c'f-cf'=0$ and $f$ is unique up to a scalar.
\end{proof}

\subsection{Differential-operator representations}

Let $s_\alpha\in W$ be the simple reflection corresponding to $\alpha\in\Delta$. If $f\in\caA_1$ is weighted with weight $\mu$, we define
\begin{equation}\label{eqact0}
s_\alpha(f)=\eta_\alpha^{(\mu+\rho)(H_\alpha)}(f)=\eta_\alpha^{\langle \mu+\rho,\alpha^\vee\rangle}(f).
\end{equation}
Since any $f\in\caA_1$ can be written as $f=\sum_{i\in\bbN}f_i$, with $f_i$ being weighted, we can define
\begin{equation}\label{eqact1}
s_\alpha(f)=\sum_{i\in\bbN}s_\alpha(f_i),
\end{equation}
for $\alpha\in\Delta$. This gives an action of $s_\alpha$ on $\caA_1$.

Suppose that $f\in\caA_1$ is a weighted solution of the system \ref{eqsin1} with weight $\mu$. In view of Lemma \ref{eqcom1}, we have
\begin{equation}\label{eqthm11}
\begin{aligned}
d_{\beta}(s_{\alpha}(f))=&d_{\beta}(\eta_{\alpha}^{\langle \mu+\rho,\alpha^\vee\rangle}(f))=[d_{\beta},\eta_{\alpha}^{\langle \mu+\rho,\alpha^\vee\rangle}](f)\\
=&(\langle \mu+\rho,\alpha^\vee\rangle)\delta_{\alpha,\beta}\eta_{\alpha}^{\langle \mu,\alpha^\vee\rangle}(-\langle \mu,\alpha^\vee\rangle+\zeta_{\alpha})(f)=0
\end{aligned}
\end{equation}
and
\begin{equation}\label{eqthm12}
\begin{aligned}
\zeta_\beta(s_{\alpha}(f))=&[\zeta_\beta,\eta_{\alpha}^{\langle \mu+\rho,\alpha^\vee\rangle}](f)+\eta_{\alpha}^{\langle \mu+\rho,\alpha^\vee\rangle}(\zeta_\beta(f))\\
=&(-\langle \mu+\rho,\alpha^\vee\rangle\alpha(H_\beta)+\mu(H_\beta))s_{\alpha}(f)=(s_\alpha\cdot\mu)(H_\beta)s_{\alpha}(f)
\end{aligned}
\end{equation}
for $\alpha,\beta\in\Delta$. Therefore, $s_{\alpha}(f)$ is also a weighted solution of the system \ref{eqsin1}, with weight $s_\alpha\cdot\mu$. Moreover, it follows from Lemma \ref{eqwt2} that
\begin{equation}\label{eqthm13}
T(s_{\alpha}(f))=x_\alpha^{\langle \mu+\rho,\alpha^\vee\rangle}T(f).
\end{equation}

\begin{lemma}\label{eqweyl}
Assume that $\langle\alpha,\alpha\rangle\leq\langle\beta,\beta\rangle$ for $\alpha,\beta\in\Delta$ $(\alpha\neq\beta)$. Let $m_{\alpha,\beta}$ be the smallest positive integer such that $(s_{\alpha}s_{\beta})^{m_{\alpha,\beta}}=1$. Let $c_1, c_2\in\bbC$.
\begin{itemize}
\item[$\mathrm{(\rmnum{1})}$]
If $m_{\alpha,\beta}=2$, then
\[
\eta_{\beta}^{c_2}\eta_{\alpha}^{c_1}
=\eta_{\alpha}^{c_1}\eta_{\beta}^{c_2};
\]
\item[$\mathrm{(\rmnum{2})}$]
If $m_{\alpha,\beta}=3$, then
\[
\eta_{\alpha}^{c_2}\eta_{\beta}^{c_1+c_2}\eta_{\alpha}^{c_1}
=\eta_{\beta}^{c_1}\eta_{\alpha}^{c_1+c_2}\eta_{\beta}^{c_2};
\]
\item[$\mathrm{(\rmnum{3})}$]
If $m_{\alpha,\beta}=4$, then
\[
\eta_{\beta}^{c_2}\eta_{\alpha}^{c_1+2c_2}\eta_{\beta}^{c_1+c_2}\eta_{\alpha}^{c_1}
=\eta_{\alpha}^{c_1}\eta_{\beta}^{c_1+c_2}\eta_{\alpha}^{c_1+2c_2}\eta_{\beta}^{c_2};
\]
\item[$\mathrm{(\rmnum{4})}$]
If $m_{\alpha,\beta}=6$, then
\[
\eta_{\beta}^{c_2}\eta_{\alpha}^{c_1+3c_2}\eta_{\beta}^{c_1+2c_2}
\eta_{\alpha}^{2c_1+3c_2}\eta_{\beta}^{c_1+c_2}\eta_{\alpha}^{c_1}
=\eta_{\alpha}^{c_1}\eta_{\beta}^{c_1+c_2}\eta_{\alpha}^{2c_1+3c_2}
\eta_{\beta}^{c_1+2c_2}\eta_{\alpha}^{c_1+3c_2}\eta_{\beta}^{c_2}.
\]
\end{itemize}
\end{lemma}

\begin{proof}
Here we only prove $\mathrm{(\rmnum{2})}$, since the other cases are similar. With $\alpha\neq\beta$, we can choose $\lambda\in\frh^*$ such that $\langle\lambda+\rho,\alpha^\vee\rangle=c_1$ and $\langle\lambda+\rho,\beta^\vee\rangle=c_2$. Note that $(s_{\alpha}s_{\beta})^3=1$ corresponds to $\langle\beta,\alpha^\vee\rangle=-1$. It follows from (\ref{eqthm11}) and (\ref{eqthm12}) that $s_{\alpha}s_{\beta}s_{\alpha}(1)$ and $s_{\beta}s_{\alpha}s_{\beta}(1)$ are weighted solutions of the system \ref{eqsin1}, with the same weight $s_{\alpha+\beta}\cdot\lambda$. In view of Lemma \ref{equni1} and (\ref{eqthm13}), one has $s_{\alpha}s_{\beta}s_{\alpha}(1)=s_{\beta}s_{\alpha}s_{\beta}(1)$.

First we consider the case when $c_1, c_2\in\bbN$. Then
\[
\begin{aligned}
0=&\tau^{-1}(s_{\alpha}s_{\beta}s_{\alpha}(1)-s_{\beta}s_{\alpha}s_{\beta}(1))\\
=&\tau^{-1}(\eta_{\alpha}^{c_2}\eta_{\beta}^{c_1+c_2}\eta_{\alpha}^{c_1}(1)
-\eta_{\beta}^{c_1}\eta_{\alpha}^{c_1+c_2}\eta_{\beta}^{c_2}(1))\\
=&(E_{-\alpha}^{c_2}E_{-\beta}^{c_1+c_2}E_{-\alpha}^{c_1}
-E_{-\beta}^{c_1}E_{-\alpha}^{c_1+c_2}E_{-\beta}^{c_2})v_\lambda.
\end{aligned}
\]
Keeping in mind that $M_\lambda=U(\bar\frn)v_\lambda$ is a free $U(\bar\frn)$-module, we get
\[
E_{-\alpha}^{c_2}E_{-\beta}^{c_1+c_2}E_{-\alpha}^{c_1}
-E_{-\beta}^{c_1}E_{-\alpha}^{c_1+c_2}E_{-\beta}^{c_2}=0
\]
and thus
\begin{equation}\label{eqcom2}
\eta_{\alpha}^{c_2}\eta_{\beta}^{c_1+c_2}\eta_{\alpha}^{c_1}
=\eta_{\beta}^{c_1}\eta_{\alpha}^{c_1+c_2}\eta_{\beta}^{c_2}
\end{equation}
for $c_1,c_2\in\bbN$.

In general, with Proposition \ref{prop1} and (\ref{eqet2}) in hand, we can assume that
\[
\eta_{\alpha}^{c_1}\eta_{\beta}^{c_1+c_2}\eta_{\alpha}^{c_2}
=\sum_{p,q\in\bbN} x_{\alpha}^{c_1+c_2-p}x_{\beta}^{c_1+c_2-q}
\sum_{a\in\Gamma_s,b\in\Gamma}P_{p,q,a,b}(c_1,c_2)x^{a}\pt^b,
\]
and
\[
\eta_{\beta}^{c_2}\eta_{\alpha}^{c_1+c_2}\eta_{\beta}^{c_1}
=\sum_{p,q\in\bbN} x_{\alpha}^{c_1+c_2-p}x_{\beta}^{c_1+c_2-q}
\sum_{a\in\Gamma_s,b\in\Gamma}Q_{p,q,a,b}(c_1,c_2)x^{a}\pt^b,
\]
where $P_{p,q,a,b}(c_1,c_2)$ and $Q_{p,q,a,b}(c_1,c_2)$ are polynomials in $c_1$ and $c_2$. In view of (\ref{eqcom2}), one has
\[
P_{p,q,a,b}(c_1,c_2)-Q_{p,q,a,b}(c_1,c_2)=0
\]
for $(c_1,c_2)\in\bbN^2$. We can use induction on $l$ to show that a polynomial function on $\bbC^l$ vanishing on $\bbN^l$ must be zero. Then (\rmnum{2}) follows from the case $l=2$.
\end{proof}

\begin{theorem}\label{thm1}
The space $\caA_1$ of truncated-up formal power series is a representation of the Weyl group $W$, given by $(\ref{eqact1})$.
\end{theorem}
\begin{proof}
Let $f\in\caA_1$ be weighted function with weight $\mu$. Note that
\[
s_\alpha^2(f)=s_\alpha(\eta_\alpha^{\langle\mu+\rho,\alpha^\vee\rangle}(f))
=\eta_\alpha^{-\langle\mu+\rho,\alpha^\vee\rangle}
(\eta_\alpha^{\langle\mu+\rho,\alpha^\vee\rangle}(f))=f
\]
by (\ref{eqact0}). In view of (\ref{eqact1}),
\begin{equation*}\label{eqthm14}
(s_\alpha|_{\caA_1})^2=1_{\caA_1}\quad\mbox{for}\ \alpha\in\Delta.
\end{equation*}
It follows from Lemma \ref{eqweyl} that
\begin{equation*}\label{eqthm15}
((s_\alpha|_{\caA_1})(s_\beta|_{\caA_1}))^{m_{\alpha,\beta}}=1_{\caA_1}
\quad\mbox{for}\ \alpha,\beta\in\Delta,
\end{equation*}
where $m_{\alpha,\beta}$ is the smallest positive integer such that $(s_{\alpha}s_{\beta})^{m_{\alpha,\beta}}=1$ (the order $m_{\alpha,\beta}=2, 3, 4$ or $6$ is well known, see for example Proposition 2.8 in \cite{H2}). By Theorem 1.9 in \cite{H2}, this proves that (\ref{eqact1}) defines a representation of $W$ on $\caA_1$.
\end{proof}

%
%
\section{Solutions of the system \ref{eqsin1}}
%
%
In this section we will solve the system \ref{eqsin1} of partial differential equations. We also provide a new proof of the well-known BGG-Verma Theorem and determine all the polynomial solutions of the system.

\subsection{Weighted solutions in $\caA_1$}
Let $Z(\frg)$ be the center of the universal enveloping algebra $U(\frg)$. Recall that $M_\lambda$ has an infinitesimal character $\chi_\lambda$, where $\chi_\lambda$ is a homomorphism $Z(\frg)\ra\bbC$ so that $z\cdot v=\chi_\lambda(z)v$ for all $z\in Z(\frg)$ and $v\in M_\lambda$.

\begin{lemma}\label{eqcen1}
For $z\in Z(\frg)$, we have $z|_{\caA_1}=\chi_\lambda(z).$
\end{lemma}
\begin{proof}
First, for any $f, g\in\caA$, one has
\[
[z|_{\caA_1},f](g)=\tau(z\cdot\tau^{-1}(fg))-f\tau(z\cdot\tau^{-1}(g))=\chi_\lambda(z)fg-f\chi_\lambda(z)g=0.
\]
So $[z|_{\caA_1},f]=0$ for any $f\in\caA$. With Proposition \ref{prop1}, we can assume that $z|_{\caA_1}=\sum_{a,b\in\Gamma}C_{z,a,b}x^a\pt^b$, where $C_{z,a,b}$ are constants and nonzero for only finitely many pairs $(a, b)$. Then
\[
\begin{aligned}
0=&[z|_{\caA_1},x_\beta]=(z|_{\caA_1})x_\beta-x_\beta (z|_{\caA_1})\\
=&\sum_{a,b\in\Gamma}C_{z,a,b}(x^{a+\eps_\beta}\pt^b+b_\beta x^a\pt^{b-\eps_\beta})
-\sum_{a,b\in\Gamma}C_{z,a,b}x^{a+\eps_\beta}\pt^b\\
=&\sum_{a,b\in\Gamma}b_\beta C_{z,a,b}x^a\pt^{b-\eps_\beta}
\end{aligned}
\]
for $\beta\in\Phi^+$. So $C_{z,a,b}=0$ unless $b_\beta=0$. Therefore $z|_{\caA_1}=\sum_{a\in\Gamma}C_{z,a,0}x^a\in\caA$. On the other hand, $(z|_{\caA_1})(1)=\tau(zv_\lambda)=\chi_\lambda(z)$. Hence $z|_{\caA_1}=\chi_\lambda(z).$
\end{proof}

\begin{lemma}\label{eqcen2}
Let $f\in\caA_1$ be a nonzero weighted solution of the system $\ref{eqsin1}$, with weight $\mu$.  Then we have an isomorphism
\[
\psi:M_\mu\simeq U(\frg)|_{\caA_1}(f).
\]
given by $\psi(v_\mu)=f$. Moreover, there exists $w\in W$ such that $\mu=w\cdot \lambda$.
\end{lemma}

\begin{proof}
Since $d_\alpha(f)=0$ for $\alpha\in\Delta$, the module $U(\frg)|_{\caA_1}(f)$ is a highest weight module and
there exists a natural surjective homomorphism
\[
\psi:M_\mu\ra U(\frg)|_{\caA_1}(f)
\]
given by $\psi(v_\mu)=f$. So
\[
U(\frg)|_{\caA_1}(f)\simeq U(\bar\frn)|_{\caA_1}(f)
\]
as vector spaces. Suppose that we have $u\in U(\bar\frn)$ such that $(u|_{\caA_1})(f)=0$. Since $u|_{\caA_1}$ is generated by $\eta_\beta$ for $\beta\in\Phi^+$, it follows from Proposition \ref{prop1} that
\[
0=T((u|_{\caA_1})(f))=T(\tau(u))T(f).
\]
But $T(f)\neq0$, one obtains $T(\tau(u))=0$ and thus $u=0$. Hence $\psi$ is also injective.

It is evident that the infinitesimal character of $M_\mu\simeq U(\frg)|_{\caA_1}(f)$ is $\chi_\mu$. By Lemma \ref{eqcen1}, we obtain $\chi_\mu=\chi_\lambda$. Then $\mu\in W\cdot\lambda$ by the Harish-Chandra isomorphism.
\end{proof}

Now we can state our first main result in this section.

\begin{theorem}\label{thm2}
The solution space of the system $\ref{eqsin1}$ in $\caA_1$ is spanned by $\{w(1)\ |\ w\in W\}$, which is the set of all the weighted solutions of the system $\ref{eqsin1}$ up to scalar multiples.
\end{theorem}

\begin{proof}
For $w\in W$, it follows from (\ref{eqthm11}) and (\ref{eqthm12}) that $w(1)$ is a weighted solution of the system \ref{eqsin1}, with weight $w\cdot\lambda$. On the other hand, Lemma \ref{equni1} and Lemma \ref{eqcen2} imply that $\{w(1)\ |\ w\in W\}$ is the set of all the weighted solutions of the system \ref{eqsin1} up to scalar multiples.
\end{proof}

%
%
\subsection{BGG-Verma Theorem and polynomial solutions}
%
%

Recall that any nonzero weighted polynomial solution of the system \ref{eqsin1} corresponds to a singular vector in $M_\lambda$, and thus a homomorphism between Verma modules. With Lemma \ref{equni1} and Lemma \ref{eqcen2} in hand, we can recover the following well-known result due to Verma \cite{V}.

\begin{theorem}\label{Verma1}
Let $\vf:M_\mu\ra M_\lambda$ be a nonzero homomorphism for $\lambda,\mu\in\frh^*$.
\begin{itemize}
\item[($\mathrm{\rmnum{1}}$)] In all cases, $\dim\Hom_{\frg}(M_\mu,M_\lambda)\leq1$;
\item[($\mathrm{\rmnum{2}}$)] The homomorphism $\vf$ is injective;
\item[($\mathrm{\rmnum{3}}$)] There exists $w\in W$ such that $\mu=w\cdot \lambda$.
\end{itemize}
\end{theorem}

It follows from the above theorem that we can write $M_\mu\subset M_\lambda$ whenever $\Hom_\frg(M_\mu, M_\lambda)\neq0$.

\begin{lemma}\label{eqcom3}
Let $u\in U(\bar\frn)$, $\alpha\in\Delta$ and $s\in\bbC$. Then there exists $c\in\bbC$ such that $\eta_\alpha^{c}(u|_{\caA_1})\in U(\bar\frn)|_{\caA_1}\eta_\alpha^{s}$ and $c-s\in\bbN$.
\end{lemma}

\begin{proof}
In view of Lemma \ref{commutator}, there exists $u'(c)\in U(\bar\frn)$ such that
\begin{equation}\label{eqcom31}
\eta_\alpha^{c}\eta_\beta=(u'(c)|_{\caA_1})\eta_\alpha^{c-3}
\end{equation}
for each $\beta\in\Phi^+$ and $c\in\bbN$. The coefficients of $u'(c)$ in the PBW basis depend polynomially on $c$. In the spirit of Lemma \ref{eqcom1}, we can show that (\ref{eqcom31}) is also true for $c\in\bbC$. Then the assertion follows from repeated applications of (\ref{eqcom31}).
\end{proof}

\begin{prop}\label{prop2}
Let $\lambda,\mu\in\frh^*$ and $\alpha\in\Delta$. Denote $s=\langle\lambda+\rho,\alpha^\vee\rangle$ and $t=\langle\mu+\rho,\alpha^\vee\rangle$. Suppose that $M_\mu\subset M_\lambda$.
\begin{itemize}
\item[$\mathrm{(\rmnum{1})}$] If $t\in\bbZ^{\leq0}$, then $M_{\mu}\subset M_{s_\alpha\cdot\lambda}.$
\item[$\mathrm{(\rmnum{2})}$] If $t\not\in\bbZ^{\leq0}$, then $M_{s_\alpha\cdot\mu}\subset M_{s_\alpha\cdot\lambda}.$
\end{itemize}
\end{prop}

\begin{proof}
Sicne $M_\mu\subset M_\lambda$, we have $s-t=\langle\lambda-\mu,\alpha^\vee\rangle\in\bbZ$. Moreover, there exists $u\in U(\bar\frn)$ such that $v_\mu=u v_\lambda$. Then $\tau(v_\mu)=\tau(uv_\lambda)=u|_{\caA_1}(1)$ is a polynomial solution of the system $\ref{eqsin1}$, with weight $\mu$. The above Lemma insures the existence of $c\in\bbC$ such that $c-s\in\bbN$ and
$\eta_\alpha^{c}(u|_{\caA_1})\in U(\bar\frn)|_{\caA_1}\eta_\alpha^{s}$. By Lemma \ref{eqcen2}, we can identify $U(\frg)|_{\caA_1}\eta_\alpha^{s}(1)$ with $M_{s_\alpha\cdot\lambda}$ as $U(\frg)$-modules. Therefore
\begin{equation}\label{eqprop20}
\begin{aligned}
\eta_\alpha^{c}(u|_{\caA_1}(1))\in U(\bar\frn)|_{\caA_1}\eta_\alpha^{s}(1)\subset M_{s_\alpha\cdot\lambda}.
\end{aligned}
\end{equation}
It is easy to see that we can increase $\Ree(c)$ in Lemma \ref{eqcom3}, such that $\Ree(c)\geq0$ and $c-t=(c-s)+(s-t)\in\bbN$. It follows from Lemma \ref{eqcom1} that
\begin{equation}\label{eqprop21}
\begin{aligned}
\left[d_\alpha, \eta_\alpha^{c}\right](u|_{\caA_1}(1))
=&c\eta_\alpha^{c-1}(1-c+\zeta_\alpha)(u|_{\caA_1}(1))\\
=&c(t-c)\eta_\alpha^{c-1}(u|_{\caA_1}(1)).
\end{aligned}
\end{equation}
Since $u|_{\caA_1}(1)$ is a solution of the system, the bracket on the left is
\begin{equation}\label{eqprop22}
d_\alpha\eta_\alpha^{c}(u|_{\caA_1}(1))-\eta_\alpha^{c}d_\alpha (u|_{\caA_1}(1))
=d_\alpha\eta_\alpha^{c}(u|_{\caA_1}(1))\in M_{s_\alpha\cdot\lambda}.
\end{equation}

For (\rmnum{1}), since $\Ree(c)\geq0$, we have $c=(c-t)+t\in\bbN$. If $c=0$, then $u|_{\caA_1}(1)\in M_{s_\alpha\cdot\lambda}$ by (\ref{eqprop20}). Thus Lemma \ref{eqcen2} yields the desired embedding $M_{\mu}\subset M_{s_\alpha\cdot\lambda}.$ If $c>0$, then $t-c\leq-c<0$. Hence (\ref{eqprop21}) and (\ref{eqprop22}) imply that $$\eta_\alpha^{c-1}(u|_{\caA_1}(1))\in M_{s_\alpha\cdot\lambda}.$$ So we can reduce $c$ stepwise and eventually get $u|_{\caA_1}(1)\in M_{s_\alpha\cdot\lambda}$.

For (\rmnum{2}), since $c-t\in\bbN$, we obtain $c=(c-t)+t\neq0$ when $t\not\in\bbZ^{\leq0}$. If $c-t=0$, then $\eta_\alpha^{t}(u|_{\caA_1}(1))\in M_{s_\alpha\cdot\lambda}$ by (\ref{eqprop20}). Again by Lemma \ref{eqcen2}, we have $M_{s_\alpha\cdot\mu}\subset M_{s_\alpha\cdot\lambda}.$ If $c-t>0$, we can get $\eta_\alpha^{c-1}(u|_{\caA_1}(1))\in M_{s_\alpha\cdot\lambda}$ and thus eventually get $\eta_\alpha^{t}(u|_{\caA_1}(1))\in M_{s_\alpha\cdot\lambda}$ by a similar reduction.
\end{proof}

Recall that there exists a natural lexicographic ordering ``$\prec$" on $\bbC$, that is, we write $c_1\prec c_2$ if $\Ree(c_1-c_2)<0$ or $\Ree(c_1-c_2)=0$ and $\Imm(c_1-c_2)\leq0$ for $c_1, c_2\in\bbC$. With this ordering on $\bbC$, we can also define a natural partial ordering on $\frh^*$. Write $\lambda\prec\mu$ if and only if $\lambda-\mu=\sum_{\alpha\in\Delta}k_\alpha\alpha$, with all $k_\alpha\prec0$. Call $\lambda\in\frh^*$ {\it strictly antidominant} if $\langle\lambda+\rho,\beta^\vee\rangle\prec 0$ for all $\beta\in\Phi^+$. A lemma follows immediately from the definitions.

\begin{lemma}\label{eqantido}
For $\lambda\in\frh^*$, the following conditions are equivalent:
\begin{itemize}
\item[$\mathrm{(\rmnum{1})}$] $\lambda$ is strictly antidominant;
\item[$\mathrm{(\rmnum{2})}$] $\lambda-s_\alpha\cdot\lambda\prec0$ for all $\alpha\in\Delta$;
\item[$\mathrm{(\rmnum{3})}$] $\lambda\prec w\cdot\lambda\prec w_0\cdot\lambda$ for all $w\in W$, where $w_0$ is the longest element in $W$.
\end{itemize}
Therefore there exists the unique strictly antidominant weight in the orbit $W\cdot\lambda$.
\end{lemma}

Given $\lambda,\mu\in\frh^*$, write $\mu\uparrow\lambda$ if there exists $\gamma\in\Phi^+$ such that $\mu=s_\gamma\cdot\lambda$ and
$\langle\lambda+\rho,\gamma^\vee\rangle\in\mathbb{Z}^{>0}$. In general, if $\mu=\lambda$ or there exist $\gamma_1,
\ldots, \gamma_r\in\Phi^+$ such that
\[
\mu=(s_{\gamma_1}\ldots
s_{\gamma_r})\cdot\lambda\uparrow(s_{\gamma_2}\ldots
s_{\gamma_r})\cdot\lambda\uparrow\ldots\uparrow
s_{\gamma_r}\cdot\lambda\uparrow\lambda.
\]
we say that $\mu$ is {\it strongly linked} to $\lambda$ by $\gamma_1,
\ldots, \gamma_r$ and
write $\mu\uparrow\lambda$.

With Proposition \ref{prop2} in hand,  we can give an elementary approach to the following well-known result of Verma, BGG and Jantzen \cite{V,BGG1,J}.

\begin{theorem}\label{BGG-V thm}
Let $\lambda,\mu\in\frh^*$.
Then $M_\mu\subset M_\lambda$ if and only if $\mu$
is strongly linked to $\lambda$.
\end{theorem}

\begin{proof}
By Lemma \ref{eqantido}, let $\lambda_0\in\frh^*$ be the unique strictly antidominant weight in $W\cdot\lambda$. There exists $w,w'\in W$ such that $\lambda=w\cdot\lambda_0$ and $\mu=w'\cdot\lambda_0$. Once $\alpha\in\Delta$ is fixed in the context, denote $s=\langle\lambda+\rho,\alpha^\vee\rangle$ and $t=\langle\mu+\rho,\alpha^\vee\rangle$.

(1) In one direction, assume that $M_\mu\subset M_\lambda$. Then $s-t=\langle\lambda-\mu,\alpha^\vee\rangle\in\bbZ$. The proof goes by induction on $l(w)$, where $l(w)$ is the usual length function on $W$. If $\lambda=\lambda_0$ (e.g., $w=1$), then $\lambda\prec w'\cdot\lambda=\mu$ by Lemma \ref{eqantido}. On the other hand, $M_\mu\subset M_\lambda$ implies $\lambda-\mu\in\sum_{\alpha\in\Delta}\bbN\alpha$. We obtain $\mu\prec\lambda$ and thus $\lambda=\mu$. Now suppose that $\lambda=w\cdot\lambda_0$ is not strictly antidominant. Then there exists $\alpha\in\Delta$ such that $0\prec\lambda -s_\alpha\cdot\lambda\neq0$, that is,
\begin{equation}\label{eqBGGV1}
0\prec s=\langle\lambda+\rho,\alpha^\vee\rangle
=\langle\lambda_0+\rho,(w^{-1}\alpha)^\vee\rangle\neq0.
\end{equation}
It turns out $w^{-1}\alpha<0$ and thus $l(s_\alpha w)<l(w)$. In view of Proposition \ref{prop2}, if $t\in\bbZ^{\leq0}$, then $M_\mu\subset M_{s_\alpha\cdot\lambda}$. The induction hypothesis can be applied to $l(s_\alpha w)$, showing that $\mu\uparrow s_\alpha\cdot\lambda$. Moreover, since $s=(s-t)+t\in\bbZ$, we obtain $s\in\bbN$ and $s_\alpha\cdot\lambda\uparrow\lambda$ by (\ref{eqBGGV1}). Therefore one has $\mu\uparrow\lambda$. If $t\not\in\bbZ^{\leq0}$, then $M_{s_\alpha\cdot\mu}\subset M_{s_\alpha\cdot\lambda}$. Thanks to the induction hypothesis, we can assume that $s_\alpha\cdot\mu\uparrow s_\alpha\cdot\lambda$ by $\gamma_1,\ldots, \gamma_k$. If $\gamma_i\neq\alpha$ for $i=1,\ldots,k$, then $\mu\uparrow\lambda$ by $s_\alpha\gamma_1,\ldots,s_\alpha\gamma_k$. Otherwise let $i_0$ be the smallest positive integer such that $\gamma_{i_0}=\alpha$. Denote $\nu=(s_{\gamma_{i_0+1}}\cdots s_{\gamma_k})\cdot(s_\alpha\cdot\lambda)$. Then $\langle\nu+\rho,\alpha^\vee\rangle\in\bbN$ and $\langle s_\alpha\cdot\lambda-\nu,\alpha^\vee\rangle\in\bbZ$ since $s_\alpha\cdot\mu\uparrow s_\alpha\cdot\nu\uparrow\nu\uparrow s_\alpha\cdot\lambda$. It follows that $\langle\lambda+\rho,\alpha^\vee\rangle\in\bbZ$. In view of (\ref{eqBGGV1}), we obtain $\langle\lambda+\rho,\alpha^\vee\rangle\in\bbN$ and thus $\mu\uparrow\nu\uparrow\lambda$ by $s_\alpha\gamma_1,\ldots,s_\alpha\gamma_{i_0-1},\gamma_{{i_0}+1},\ldots,\gamma_k,\alpha$.

(2) In the other direction, assume that $\mu\uparrow\lambda$ by $\gamma_1, \ldots, \gamma_k$. It is obvious that $\mu\prec\lambda$. Now use downward induction on $l(w')$, starting with $\mu=w_0\cdot\lambda_0$ (e.g., $w'=w_0$). Then $\lambda\prec\mu$ by Lemma \ref{eqantido}. One must have $\lambda=\mu$ and $M_\mu= M_\lambda$. Now suppose that $\mu\neq w_0\cdot\lambda_0$. Then there exists $\alpha\in\Delta$ such that
\begin{equation}\label{eqBGGV2}
0\neq t=\langle\mu+\rho,\alpha^\vee\rangle=\langle\lambda_0+\rho,(w'^{-1}\alpha)^\vee\rangle\prec0.
\end{equation}
We get $w'^{-1}\alpha>0$ and thus $l(s_\alpha w')>l(w')$. If $\gamma_i\neq\alpha$ for $i=1,\ldots,k$, then $s_\alpha\cdot\mu\uparrow s_\alpha\cdot\lambda$ by $s_\alpha\gamma_1,\ldots,s_\alpha\gamma_k$. By the induction hypothesis, we get $M_{s_\alpha\cdot\mu}\subset M_{s_\alpha\cdot\lambda}$. Since (\ref{eqBGGV2}) implies $\langle s_\alpha\cdot\mu+\rho,\alpha^\vee\rangle=-t\not\in\bbZ^{\leq0}$, it follows from Proposition \ref{prop2} that $M_\mu\subset M_\lambda$. Otherwise denote by $i_0$ the smallest positive integer satisfying $\gamma_{i_0}=\alpha$. Denote $\nu=(s_{\gamma_{i_0+1}}\cdots s_{\gamma_k})\cdot\lambda$. Then $\langle\nu+\rho,\alpha^\vee\rangle\in\bbN$ and $\langle \nu-\mu,\alpha^\vee\rangle\in\bbZ$ since $\mu\uparrow s_\alpha\cdot\nu\uparrow\nu$.
It follows that $t=\langle\mu+\rho,\alpha^\vee\rangle\in\bbZ$. With (\ref{eqBGGV2}), one has $-t\in\bbN$. Then $M_{\mu}\subset M_{s_\alpha\cdot\mu}$ since $E_{-\alpha}^{-t}v_{s_\alpha\cdot\mu}$ is a singular vector in $M_{s_\alpha\cdot\mu}$ of weight $\mu$. On the other hand, since $s_\alpha\cdot\mu\uparrow\nu
\uparrow\lambda$ by $s_\alpha\gamma_1,\ldots,s_\alpha\gamma_{i_0-1},\gamma_{{i_0}+1},\ldots,\gamma_k$, the induction hypothesis can be applied, yielding $M_{s_\alpha\cdot\mu}\subset M_\lambda$ and thus $M_{\mu}\subset M_\lambda$.

\end{proof}

\begin{remark}
The sufficiency of the strongly linked condition was discovered by Verma using enumeration of simple Lie algebras \cite{V}. BGG showed the necessity and reproduced the sufficiency by a different method \cite{BGG1}. The necessity can also be proved by the Jantzen filtration and contravariant forms \cite{J}. Their arguments rely on some deep results. The proof in this section seems more ``elementary".
\end{remark}

One immediate consequence of Theorem \ref{BGG-V thm} is

\begin{theorem}\label{thm3}
Let $f\in\caA_1$ be a weighted solution of the system $\ref{eqsin1}$, with weight $\mu$. Then $f$ is a polynomial if and only if $\mu$ is strongly linked to $\lambda$.
\end{theorem}

%
%
\section{Singular vectors of $\frsp(2n)$}
%
%
 Using the notation and results in \cite{Xu2}, we will explore the applications of our main theorems for $\frsp(2n)$ in this section. The case of $\frsl(n,\bbC)$ has been well studied in \cite{Xu1}. Some partial results for $\frsp(2n)$ were given in \cite{Xu2}. With Theorem \ref{thm2} and Theorem \ref{thm3} in hand, a natural goal is to explicitly write down singular vectors in Verma module $M_\lambda$. We did this for $\frsl(n, \bbC)$ in \cite{Xi1}. We will give a similar result for $\frsp(2n)$ in this section.

Let $E_{i,j}$ be the $2n\times 2n$ matrix with $1$ in the $(i,j)$ position and $0$ elsewhere. The symplectic Lie algebra
\[
\begin{aligned}
\frsp(2n)=&\sum_{i,j=1}^n\bbC(E_{i,j}-E_{n+j,n+i})+\sum_{i=1}^n(\bbC E_{i,n+i}+\bbC E_{n+i,i})\\
&+\sum_{1\leq i<j\leq n}[\bbC(E_{i,n+j}+E_{j,n+i})+\bbC(E_{n+i,j}+E_{n+j,i})]
\end{aligned}
\]
is a Lie subalgebra of the special linear algebra $\frsl(2n,\bbC)$. Denote
\[
H_i=E_{i,i}-E_{i+1,i+1}-E_{n+i,n+i}+E_{n+i+1,n+i+1}
\]
for $i=1,2,\ldots,n-1$ and $H_n=E_{n,n}-E_{2n,2n}$. Then
\[
\frh=\sum_{i=1}^n\bbC H_i
\]
is a Cartan subalgebra of $\frsp(2n)$. Choose positive root vectors
\[
\{E_{i,j}-E_{n+j,n+i}, E_{i,n+j}+E_{j,n+i}, E_{k,n+k}\ |\ 1\leq i<j\leq n;\ 1\leq k\leq n\}
\]
and negative root vectors
\[
\{E_{i,j}-E_{n+j,n+i}, E_{n+i,j}+E_{n+j,i}, E_{n+k,k}\ |\ 1\leq j<i\leq n;\ 1\leq k\leq n\}.
\]
Let $e_i$ be the linear function on $\frh$ such that
\[
e_i(E_{j,j}-E_{n+j,n+j})=\delta_{ij}.
\]
Then the corresponding positive roots are
\[
\Phi^+=\{e_i-e_j, e_i+e_j, 2e_k\ |\ 1\leq i<j\leq n;\ 1\leq k\leq n\}.
\]
For convenience, we set
\[
C_{i,j}=E_{i,j}-E_{n+j,n+i},\ C_{i,n+j}=E_{i,n+j}+E_{j,n+i}, \ C_{n+i,j}=E_{n+i,j}+E_{n+j,i}
\]
for $1\leq i, j\leq n$ with $i\neq j$, and
\[
C_{n+k,k}=E_{n+k,k},\  C_{k,n+k}=E_{k,n+k}
\]
for $1\leq k\leq n$. Here $C_{i,n+j}=C_{j,n+i}$ and $C_{n+i,j}=C_{n+j,i}$. Set
\[
\Gamma:=\sum_{1\leq j<i\leq n}\bbN \eps_{i,j}+\sum_{1\leq j\leq i\leq n}\bbN\eps_{n+i,j}
\]
with base elements $\eps_{p,q}$. Given
\[
a=\sum_{1\leq j<i\leq n}a_{i,j} \eps_{i,j}+\sum_{1\leq j\leq i\leq n} a_{n+i,j}\eps_{n+i,j}\in\Gamma,
\]
denote
\begin{equation}\label{eq PBW basis2}
\begin{aligned}
E^a=&C_{2,1}^{a_{2,1}}C_{3,1}^{a_{3,1}}C_{3,2}^{a_{3,2}}C_{4,1}^{a_{4,1}}\ldots C_{n,1}^{a_{n,1}}\ldots C_{n,n-1}^{a_{n,n-1}}\\
&\times C_{n+1,1}^{a_{n+1,1}}C_{n+2,1}^{a_{n+2,1}}C_{n+2,2}^{a_{n+2,2}}C_{n+3,1}^{a_{n+3,1}}\ldots C_{2n,1}^{a_{2n,1}}\ldots C_{2n,n}^{a_{2n,n}}.
\end{aligned}
\end{equation}
and
\[
x^a=\prod_{1\leq j<i\leq n}x_{i,j}^{a_{i,j}}\prod_{1\leq j\leq i\leq n}x_{n+i,j}^{a_{n+i,j}}.
\]
In particular, $\{x^a\ |\ a\in\Gamma\}$ form a basis of the polynomial algebra
\[
\caA=\bbC[x_{i,j},x_{n+i,j},x_{n+k,k}\ |\ 1\leq j<i\leq n;1\leq k\leq n].
\]
Denote by $\pt_{i,j}$ the partial derivative $\pt_{x_{i,j}}$ for simplicity. For convenience, we write
\[
x_{n+i,j}=x_{n+j,i},\quad\pt_{n+i,j}=\pt_{n+j,i}
\]
for $1\leq i<j\leq n$. Recall that a weight $\lambda\in\frh^*$ is a linear function on $\frh$. Denote
\[
\lambda_i:=\lambda(H_i)+1
\]
for $i=1,2,\ldots,n$. Although $\lambda_i$ is assigned to $\lambda(H_i)$ in \cite{Xu2}, we find that the formula of singular vectors could be effectively simplified if $\lambda_i$ is defined to be $\lambda(H_i)+1$. It can be shown (see \cite{Xu2}, 2.40, 2.41) in this setting that
\[
\begin{aligned}
d_i=&C_{i,i+1}|_{\caA}\\
=&\Big(\lambda_i-1-\sum_{j=i+1}^nx_{j,i}\pt_{j,i}+\sum_{j=i+2}^nx_{j,i+1}\pt_{j,i+1}+\sum_{k\neq i,i+1}(x_{n+k,i+1}\pt_{n+k,i+1}-x_{n+k,i}\pt_{n+k,i})\\
&-2x_{n+i,i}\pt_{n+i,i}+2x_{n+i+1,i+1}\pt_{n+i+1,i+1}\Big)\pt_{i+1,i}+\sum_{j=1}^{i-1}x_{i,j}\pt_{i+1,j}\\
&-\sum_{j=i+2}^nx_{j,i+1}\pt_{j,i}-\sum_{k\neq i+1}x_{n+k,i+1}\pt_{n+k,i}-2x_{n+i+1,i+1}\pt_{n+i+1,i}
\end{aligned}
\]
for $i=1,2,\ldots,n-1$ and
\[
d_n=C_{n,2n}|_\caA=(\lambda_n-1-x_{2n,n}\pt_{2n,n})\pt_{2n,n}+\sum_{i=1}^{n-1}\left(x_{n,i}+\sum_{j=1}^ix_{n+i,j}\pt_{2n,j}\right)\pt_{2n,i}
\]
\begin{prop}[See \cite{Xu2}, Proposition\ 2.1]
Let $u$ be a weight vector in $M_\lambda$, then $u$ is a singular vector if and only if
\begin{equation}\label{eqsp1}
d_i(\tau(u))=0
\end{equation}
for $i=1,2,\dots,n$.
\end{prop}

Moreover, we have (see \cite{Xu2}, 3.3, 3.4)
\begin{equation}\label{etai}
\eta_i:=\eta_{e_i-e_{i+1}}=C_{i+1,i}|_{\caA}=x_{i+1,i}+\sum_{j=1}^{i-1}x_{i+1,j}\pt_{i,j}
\end{equation}
for $i=1,2,\ldots,n-1$ and
\begin{equation}\label{etan}
\eta_n:=\eta_{2e_n}=C_{2n,n}|_\caA=x_{2n,n}+\sum_{j=1}^{n-1}(x_{2n,j}+\sum_{i=j}^{n-1}x_{n+i,j}\pt_{n,i})\pt_{n,j}.
\end{equation}

Let $\iota$ be the ordering of $\Phi^+$ corresponding to (\ref{eq PBW basis2}). It is not difficult to check that $\iota$ is a good ordering. Note that in this case,
\[
\caA_0=\bbC[x_{i,j}, x_{n+p, q}\ |\ 1\leq j< i-1\leq n-1;\ 1\leq q\leq p\leq n;\ q\neq n].
\]
Let $\caA_1$ be the space of truncated-up formal power series in $\{x_{2,1},\ldots, x_{n,n-1}, x_{2n,n}\}$ over $\caA_0$. We can also define the action of $W$ on $\caA_1$ as in the previous sections. Then the space $\caA_1$ is a differential-operator representation of the Weyl group $W$. The solution space of the system \ref{eqsp1} in $\caA_1$ is the span of $\{w(1)\ |\ w\in W\}$, which is the set of weighted solutions  up to scalars. The function $w(1)$ is a polynomial if and only if $w\cdot\lambda$ is strongly linked to $\lambda$.

\subsection{Formula of singular vectors} With the setting in this section, some beautiful formulas of singular vectors can be obtained. For $k\in\bbN$ and $\beta\in\Phi^+$, denote
\[
\Gamma_{\beta}^k:=\{a\in\Gamma\ |\ [H, E^a]=k\beta(H)E^a\ \mbox{for all}\ H\in\frh\}.
\]

\begin{example}
Note that
\[
\begin{aligned}
&s_{e_1+e_2}(1)\\
=&s_{2e_2}s_{e_1-e_2}s_{2e_2}(1)=\eta_2^{\lambda_1+\lambda_2}
(x_{2,1}^{\lambda_1+2\lambda_2}x_{4,2}^{\lambda_2})\\
=&(x_{4,2}+x_{4,1}\pt_{2,1}+x_{3,1}\pt_{2,1}^2)^{\lambda_1+\lambda_2}(x_{2,1}^{\lambda_1+2\lambda_2}x_{4,2}^{\lambda_2})\\
=&\sum_{p, q\in\bbN}\frac{\langle\lambda_1+\lambda_2\rangle_{p+q}}{p!q!}
x_{4,2}^{\lambda_1+\lambda_2-p-q}(x_{4,1}\pt_{2,1})^{q}
(x_{3,1}\pt_{2,1}^2)^{p}x_{2,1}^{\lambda_1+2\lambda_2}x_{4,2}^{\lambda_2}\\
=&\sum_{p, q\in\bbN}\frac{\langle\lambda_1+\lambda_2\rangle_{p+q}\langle\lambda_1+2\lambda_2\rangle_{2p+q}}
{p!q!}
x_{2,1}^{\lambda_1+2\lambda_2-2p-q}x_{3,1}^{p}x_{4,1}^{q}x_{4,2}^{\lambda_1+2\lambda_2-p-q}
\end{aligned}
\]
is a solution of (\ref{eqsp1}). Suppose that $\langle\lambda+\rho,e_1+e_2\rangle=\lambda_1+2\lambda_2=k\in\bbN$. Set
\[
a=(k-2p-q)\epsilon_{2,1}+p\epsilon_{3,1}+q\epsilon_{4,1}+(k-p-q)\epsilon_{4,2},
\]
$u_1=\lambda_1+\lambda_2$ and $r_1(a)=a_{3,1}+a_{4,1}=p+q$. Then
\[
\begin{aligned}
v=&\tau^{-1}(s_{e_1+e_2}(1))\\
=&\tau^{-1}\left(\sum_{p, q\in\bbN}\frac{\langle\lambda_1+\lambda_2\rangle_{p+q}k!(k-p-q)!}
{p!q!(k-2p-q)!(k-p-q)!}
x_{2,1}^{k-2p-q}x_{3,1}^{p}x_{4,1}^{q}x_{4,2}^{k-p-q}\right)\\
=&k!\sum_{a\in\Gamma_{e_1+e_2}^k}\frac{\langle u_1\rangle_{r_1(a)}(k-r_1(a))!}
{a_{2,1}!a_{3,1}!a_{4,1}!a_{4,2}!}E^av_\lambda
\end{aligned}
\]
is a singular vector of $M_\lambda$ with weight $s_{e_1+e_2}\cdot\lambda$.
\end{example}

The above result can be generalized to any root $e_1+e_n$ for $n\geq 2$.

\begin{theorem}\label{thm4}
Suppose that $\langle\lambda+\rho,e_1+e_n\rangle=k\in\bbN$ for $\lambda\in\frh^*$. Set $u_i=\sum_{j=1}^i\lambda_j$ for $i=1, 2, \ldots, n-1$. Given $a\in\Gamma$, define
\[
r_i(a)=\sum_{j<i+1<q\leq n}a_{q,j}+\sum_{j\leq i<q\leq n}a_{n+q,j}+2\sum_{j\leq q\leq i}a_{n+q,j}
\]
for $i=1, 2, \ldots, n-2$ and $r_{n-1}(a)=\sum_{j\leq q\leq n, j\neq n}a_{n+q,j}$. Then
\begin{equation}\label{thm41}
v=k!\sum_{a\in\Gamma_{e_1+e_n}^k}\frac{\prod_{i=1}^{n-1}\langle u_i\rangle_{r_i(a)}( k-r_i(a))!}{\prod_{1\leq j<i\leq n}{a_{i,j}!}\prod_{1\leq j\leq i\leq n}{a_{n+i,j}!}}E^av_\lambda
\end{equation}
is a singular vector in $M_\lambda$ of weight $s_{e_{1}+e_{n}}\cdot\lambda$.
\end{theorem}

The following lemma is useful in the proof of the above theorem.

\begin{lemma}\label{lem4}
Let $f$ be a weighted solution of the system (\ref{eqsp1}) with weight $\mu=\lambda-2\lambda_n e_n$. We use the same notation in the above theorem. If $\eta_i(f)=x_{i+1,i}f$ for $i=1,\ldots, n-1$, then
\[
s_{e_{1}-e_{n}}(f)=f\sum_{a\in\Gamma_{e_1-e_n}^k}\frac{k!\prod_{i=1}^{n-2}\langle u_i\rangle_{r_i(a)}(k-r_i(a))!}
{\prod_{1\leq j<i\leq n}{a_{i,j}!}}x^{a}.
\]
\end{lemma}
\begin{proof}
Since $\langle\mu+\rho,e_1-e_n\rangle=k\in\bbN$, by Theorem \ref{BGG-V thm}, there exists $u\in U(\bar\frn)$ such that $uv_\mu \in M_{\mu}$ is a singular vector in $M_{\mu}$ with weight $s_{e_1-e_n}\cdot\mu$. In view of Lemma \ref{eqcen2}, we have $U(\frg)|_{\caA_1}(f)\simeq M_\mu$ and thus $u|_{\caA_1}(f)$ is also a weighted solution of the system (\ref{eqsp1}). By Lemma \ref{equni1}, we can assume that
\begin{equation}\label{eqlem41}
u|_{\caA_1}(f)=s_{e_{1}-e_{n}}(f).
\end{equation}
It follows from Theorem 4.4 in \cite{Xi2} that
\[
u=\sum_{a\in\Gamma_{e_1-e_n}^k}\frac{\prod_{i=1}^{n-1}\langle u'_i\rangle_{r'_i(a)}( k-r'_i(a))!}{\prod_{1\leq j\leq i\leq n}{a_{i,j}!}}E^a,
\]
where $u'_i=\sum_{j=1}^i(\mu(H_j)+1)$ and $r'_i(a)=\sum_{1\leq j<i+1<q\leq n}a_{q, j}$ for $i=1, \ldots, n-1$. It is evident that $u'_i=u_i$ for $i\neq n-1$ and $u'_{n-1}=k$. Since $a_{n+q,j}=0$ for $a\in\Gamma_{e_1-e_n}^k$ and $q\geq 1$, we have $r'_i(a)=r_i(a)$ for $i\neq n-1$. Moreover $r'_{n-1}(a)=0$. Thus
\begin{equation}\label{eqlem42}
\tau(uv_\lambda)=\sum_{a\in\Gamma_{e_1-e_n}^k}\frac{k!\prod_{i=1}^{n-2}\langle u_i\rangle_{r_i(a)}(k-r_i(a))!}
{\prod_{1\leq j<i\leq n}{a_{i,j}!}}x^{a}.
\end{equation}
On the other hand, if $\eta_{e_i-e_j}(f)=\eta_{e_i-e_j}(1)f$ and $\eta_{e_j-e_l}(f)=\eta_{e_j-e_l}(1)f$ for $1\leq i<j<l\leq n$, then
\[
\eta_{e_i-e_l}(f)=-[\eta_{e_i-e_j}, \eta_{e_j-e_l}](f)=\eta_{e_i-e_l}(1)f.
\]
Since
\[
\eta_{e_i-e_{i+1}}(f)=\eta_i(f)=x_{i+1,i}f=\eta_{e_i-e_{i+1}}(1)f
\]
for $i=1,\ldots,n-1$, we can get $\eta_{e_p-e_q}(f)=\eta_{e_p-e_q}(1)f$ for all pairs $(p, q)$ with $1\leq p<q\leq n$ by induction on $q-p$. Therefore $E^a|_{\caA_1}(f)=E^a|_{\caA_1}(1)f=x^af$ for $a\in\Gamma_{e_1-e_n}^k$ and thus
\begin{equation}\label{eqlem43}
u|_{\caA_1}(f)=u|_{\caA_1}(1)f=\tau(uv_\lambda)f.
\end{equation}
With (\ref{eqlem41}),  (\ref{eqlem42}) and (\ref{eqlem43}), we can obtain the desired formula for $s_{e_{1}-e_{n}}(f)$ immediately.
\end{proof}

{\bf Proof of Theorem \ref{thm4}} Since $\langle\lambda+\rho,e_1+e_n\rangle=k\in\bbN$, then $s_{e_{1}+e_{n}}(1)$ is a polynomial solution of the system (\ref{eqsp1}) by Theorem \ref{thm2} and Theorem \ref{thm3}, with weight $\lambda-k(e_1+e_2)$. We have
\begin{equation}\label{thm42}
\begin{aligned}
s_{e_{1}+e_{n}}(1)=s_{2e_n}s_{e_1-e_n}s_{2e_n}(1)=\eta_n^{k-\lambda_n}s_{e_{1}-e_{n}}(x_{2n,n}^{\lambda_n}).
\end{aligned}
\end{equation}
The weight of the solution $x_{2n,n}^{\lambda_n}$ is $\mu=\lambda-2\lambda_n e_n$ and $\eta_i(x_{2n,n}^{\lambda_n})=x_{i+1,i}x_{2n,n}^{\lambda_n}$ for $i\neq n$. By Lemma \ref{lem4}, we obtain
\begin{equation}\label{thm43}
s_{e_{1}-e_{n}}(x_{2n,n}^{\lambda_n})=
\sum_{a'\in\Gamma_{e_1-e_n}^k}\frac{k!\prod_{i=1}^{n-2}\langle u_i\rangle_{r_i(a')}(k-r_i(a'))!}
{\prod_{1\leq j<i\leq n}{a'_{i,j}!}}x^{a'+\lambda_n\epsilon_{2n,n}}.
\end{equation}
For convenience, we set $\pt_{n,n}=1$ and
\[
\Gamma'':=\{a''\in\Gamma\ |\ a''_{i,j}=0\ \mbox{if}\ j=n,\ \mbox{or}\ i\leq n\}.
\]
By (\ref{etan}) one has $\eta_n=x_{2n,n}+\sum_{j=1}^{n-1}\sum_{i=1}^nx_{n+i,j}\pt_{n,i}\pt_{n,j}$ and
\begin{equation}\label{thm44}
\eta_n^{k-\lambda_n}=\sum_{a''\in\Gamma''}\langle k-\lambda_n\rangle_{|a''|}x_{2n,n}^{k-\lambda_n-|a''|}\prod_{j=1}^{n-1}\prod_{i=j}^n
\frac{(x_{n+i,j}\pt_{n,i}\pt_{n,j})^{a''_{n+i,j}}}{a''_{n+i,j}!}.
\end{equation}
With (\ref{thm42}), (\ref{thm43}) and (\ref{thm44}) in hand, we can denote
\begin{equation}\label{thm45}
a=a'+a''-\sum_{j=1}^{n-1}\sum_{i=j}^na''_{n+i,j}(\epsilon_{n,i}+\epsilon_{n,j})
+(k-|a''|)\epsilon_{2n,n}
\end{equation}
and
\[
c_a=\langle k-\lambda_n\rangle_{|a''|}
\frac{k!\prod_{i=1}^{n-2}\langle u_i\rangle_{r_i(a')}(k-r_i(a'))!\prod_{j=1}^{n-1}\langle a'_{n,j}\rangle_{\sum_{i=1}^ja''_{n+j,i}+\sum_{i=j}^na''_{n+i,j}}
}{\prod_{1\leq j<i\leq n}{a'_{i,j}!}\prod_{j=1}^{n-1}\prod_{i=j}^n a''_{n+i,j}!}.
\]
Here $\epsilon_{n,n}=0$. Since $s_{e_1+e_n}(1)$ is a polynomial with weight $\lambda-k(e_1+e_n)$, we obtain $c_a=0$ unless $a\in\Gamma_{e_1+e_n}^k$. Moreover
\[
s_{e_1+e_n}(1)=\sum_{a\in\Gamma_{e_1+e_n}^k}c_ax^a.
\]
By (\ref{thm45}), one has
\begin{equation}\label{main3eq8}
\left\{\begin{aligned}
&a_{i,j}=a'_{i,j}\qquad\qquad\qquad\qquad\qquad\qquad\ \mbox{if}\ 1\leq j<i<n;\\
&a_{n,j}=a'_{n,j}-\sum_{i=1}^ja''_{n+j,i}-\sum_{i=j}^na''_{n+i,j}\quad\ \mbox{if}\ 1\leq j< n;\\
&a_{n+i,j}=a''_{n+i,j}\qquad\qquad\qquad\qquad\qquad\mbox{if}\ 1\leq j\leq i\leq n, j\neq n;\\
&a_{2n,n}=k-|a''|.
\end{aligned}
\right.
\end{equation}
It follows that
\[
r_i(a')=\sum_{j<i+1<q\leq n}a'_{q,j}=\sum_{j<i+1<q\leq n}a_{q,j}+\sum_{j\leq i<q\leq n}a_{n+q,j}+2\sum_{j\leq q\leq i}a_{n+q,j}=r_i(a)
\]
for $i=1, \dots, n-2$ and
\[
u_{n-1}=\sum_{j=1}^n\lambda_j=k-\lambda_n\ \mbox{and}\ r_{n-1}(a)=\sum_{j\leq q\leq n, j\neq n}a_{n+q,j}=|a''|.
\]
Therefore $a_{2n,n}=k-r_{n-1}(a)$ and
\[
\begin{aligned}
c_a&=\langle u_{n-1}\rangle_{r_{n-1}(a)}
\frac{k!\prod_{i=1}^{n-2}\langle u_i\rangle_{r_i(a)}(k-r_i(a))!
}{\prod_{1\leq j<i<n}{a_{i,j}!}\prod_{j=1}^{n-1}\prod_{i=j}^n a_{n+i,j}!}\prod_{j=1}^{n-1}\frac{\langle a'_{n,j}\rangle_{a'_{n,j}-a_{n,j}}}{a'_{n,j}!}\\
&=\langle u_{n-1}\rangle_{r_{n-1}(a)}\frac{k!\prod_{i=1}^{n-2}\langle u_i\rangle_{r_i(a)}( k-r_i(a))!}{\prod_{1\leq j<i\leq n}{a_{i,j}!}\prod_{j=1}^{n-1}\prod_{i=j}^n{a_{n+i,j}!}}\frac{(k-r_{n-1}(a))!}{a_{2n,n}!}\\
&=k!\frac{\prod_{i=1}^{n-1}\langle u_i\rangle_{r_i(a)}( k-r_i(a))!}{\prod_{1\leq j<i\leq n}{a_{i,j}!}\prod_{1\leq j\leq i\leq n}{a_{n+i,j}!}}.
\end{aligned}
\]
Hence $$v=\tau^{-1}(s_{e_1+e_n}(1))=\sum_{a\in\Gamma_{e_1+e_n}^k}c_aE^av_\lambda$$ is a singular vector in $M_\lambda$ with weight $s_{e_1+e_n}\cdot\lambda$.

%
%
\section{The proof of Lemma \ref{eqgra1}}
%
%

In this section, we will prove Lemma \ref{eqgra1}. To avoid excessive notation, here we abandon our standard conventions about the symbols $\alpha_i, \beta_j$...

\subsection{Hanging edge and central graph}
Let $\Phi_i$ be the set of positive roots in $\Phi^+$ of height $i$, and let $k_i$ be the number of roots in $\Phi_i$. Denote by $h$ the largest height of all the positive roots. Then $n=k_1\geq\ldots\geq k_h=1$ by Theorem 3.20 in \cite{H2}. Recall from section 3 that $A(\Phi)=(a_{\beta,\gamma})_{m\times(m-n)}$, where
\[
a_{\beta,\gamma}=\left\{\begin{aligned}
&N_{\beta,-\gamma}\qquad\quad\mbox{if}\ \gamma-\beta\in\Phi^+\\
&0\qquad\qquad\quad\mbox{otherwise}
\end{aligned}
\right.
\]
for $\beta\in\Phi^+,\gamma\in\Phi^+\backslash\Delta$. Consider the following submatrices
\[
A_i(\Phi):=(a_{\beta,\gamma})_{k_i\times k_{i+1}},\ \mbox{for}\ \beta\in\Phi_i, \gamma\in\Phi_{i+1}
\]
and $i=1,2,\cdots,h-1$. Since $a_{\beta,\gamma}=0$ if $\htt\gamma\leq\htt\beta$, we have the following lemma.

\begin{lemma}\label{eqgra2}
If $A_i(\Phi)$ has full rank $k_{i+1}$ for $i=1,2,\ldots,h-1$, then $A(\Phi)$ has full rank $m-n$.
\end{lemma}

Let $U$ and $V$ be subsets of $\Phi_i$ and $\Phi_{i+1}$ respectively. We introduce an undirected graph $G(U,V)$ with vertices $U\cup V$. For $\beta\in U$ and $\gamma\in V$, $(\beta,\gamma)$ is an edge of $G(U,V)$ if $\gamma-\beta\in\Phi^+$ (which means $a_{\beta, \gamma}=N_{\beta,-\gamma}\neq0$). Then $G(U,V)$ is a bipartite graph (or bigraph). It is an induced subgraph of the Hasse diagram associated with the root poset of $\Phi$. A visualization of the root poset corresponding to each Dynkin type can be found in the appendix of \cite{R}, using 3-dimensional cubes. For each graph $G(U,V)$, we can define an associated matrix
\[
A(U,V):=(a_{\beta,\gamma})_{|U|\times |V|},\ \mbox{for}\ \beta\in U, \gamma\in V,
\]
where $|U|$ and $|V|$ are numbers of roots in $U$ and $V$ respectively. It is a submatrix of $A_i(\Phi)$. Given an edge $(\beta, \gamma)$ of $G(U,V)$, we say $(\beta, \gamma)$ is a {\it hanging edge} if $\beta$ is the unique vertex in $U$ adjacent to $\gamma$ or $\gamma$ is the unique vertex in $V$ adjacent to $\beta$.

\begin{lemma}\label{eqgra3}
Let $(\beta_0, \gamma_0)$ be a hanging edge in $G(U,V)$. Then $A(U,V)$ has full rank if and only if $A(U\backslash\{\beta_0\}, V\backslash\{\gamma_0\})$ has full rank.
\end{lemma}
\begin{proof}
Without loss of generality, we assume that $\beta_0$ is the only vertex in $U$ adjacent to $\gamma_0$. Then $a_{\beta,\gamma_0}=0$ if and only if $\beta=\beta_0$. If $A(U,V)$ has full rank, then the submatrix of $A(U,V)$ by deleting the row corresponding to $\beta_0$ and the column corresponding to $\gamma_0$ (which is exactly $A(U\backslash\{\beta_0\}, V\backslash\{\gamma_0\})$) also has full rank, and vise versa.
\end{proof}

Denote $U_0:=U$ and $V_0:=V$. In general, if $(\beta_i,\gamma_i)$ is a hanging edge of $G(U_i,V_i)$, denote $U_{i+1}:=U_i\backslash\{\beta_i\}$ and $V_{i+1}:=V_i\backslash\{\gamma_i\}$ for $i=0,1$ and so on. Since $U$ and $V$ are finite sets, there always exists $l\in\bbN$ such that $G(U_l,V_l)$ has no hanging edge. Then we say that $G(U_l,V_l)$ is a {\it central graph} of $G(U,V)$.

\begin{example}\label{examposet1}
Assume that $\Phi=D_5$. Let $e_1,e_2,\ldots, e_5$ be the usual orthonormal unit vectors which form a basis of $\bbR^5$. Then the positive roots in $\Phi$ can be expressed as $\{e_i\pm e_j\ |\ 1\leq i<j\leq 5\}$, with positive simple roots $e_1-e_2$, $e_2-e_3$, $e_3-e_4$, $e_4-e_5$ and $e_4+e_5$ (represented by $1, 2, 3, 4, 5$ respectively in the figure below). The root poset of $D_5$ can be constructed using 3-dimensional cubes (see \cite{R}). The induced subgraph $G(\Phi_2,\Phi_3)$ has a hanging edge $(e_1-e_3, e_1-e_4)$. Its central graph is $G(\Phi_2\backslash\{e_1-e_3\},\Phi_3\backslash\{e_1-e_4\})$.

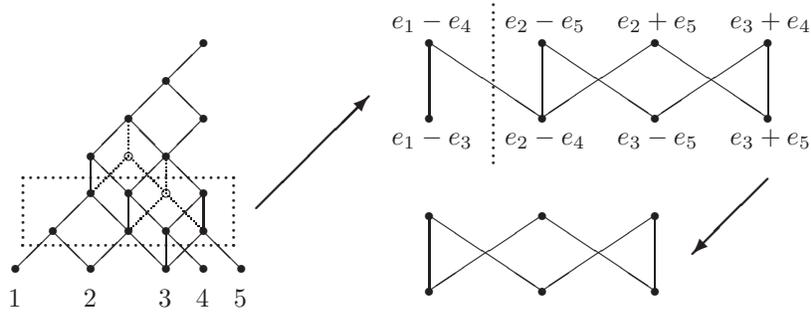
\begin{figure}[H]\label{f1}
\setlength{\unitlength}{1mm}
\begin{center}
\begin{picture}(7,37)
\put(-25,35){\circle*{1}}
\put(-30,30){\circle*{1}}
\put(-35,25){\circle*{1}} \put(-25,25){\circle*{1}}
\put(-40,20){\circle*{1}}\put(-35,20){\circle{1}}
\put(-30,20){\circle*{1}}
\put(-40,15){\circle*{1}}\put(-35,15){\circle*{1}}
\put(-30,15){\circle{1}}\put(-25,15){\circle*{1}}
\put(-45,10){\circle*{1}}\put(-35,10){\circle*{1}}
\put(-30,10){\circle*{1}}\put(-25,10){\circle*{1}}
\put(-50,5){\circle*{1}}\put(-40,5){\circle*{1}}
\put(-30,5){\circle*{1}}\put(-25,5){\circle*{1}}
\put(-20,5){\circle*{1}}

\put(-25,35){\line(-1,-1){5}}
\put(-30,30){\line(-1,-1){5}} \put(-30,30){\line(1,-1){5}}
\put(-35,25){\line(-1,-1){5}} \dottedline{0.5}(-35,25)(-35,20)
\put(-35,25){\line(1,-1){5}} \put(-25,25){\line(-1,-1){5}}
\put(-40,20){\line(0,-1){5}}\put(-40,20){\line(1,-1){5}}
\dottedline{0.5}(-40,15)(-35,20)\dottedline{0.5}(-30,15)(-35,20)
\put(-30,20){\line(-1,-1){5}}\dottedline{0.5}(-30,20)(-30,15)
\put(-30,20){\line(1,-1){5}}
\put(-40,15){\line(-1,-1){5}}\put(-40,15){\line(1,-1){5}}
\put(-35,15){\line(0,-1){5}}\put(-35,15){\line(1,-1){5}}
\dottedline{0.5}(-35,10)(-30,15)\dottedline{0.5}(-25,10)(-30,15)
\put(-25,15){\line(-1,-1){5}}\put(-25,15){\line(0,-1){5}}
\put(-45,10){\line(-1,-1){5}}\put(-45,10){\line(1,-1){5}}
\put(-35,10){\line(-1,-1){5}}\put(-35,10){\line(1,-1){5}}
\put(-30,10){\line(0,-1){5}}\put(-30,10){\line(1,-1){5}}
\put(-25,10){\line(-1,-1){5}}\put(-25,10){\line(1,-1){5}}

\put(-51,0){1}\put(-41,0){2}\put(-31,0){3}\put(-26,0){4}\put(-21,0){5}
\dottedline[$\cdot$]{1}(-21,17)(-21,8)(-49,8)(-49,17)(-21,17)

\thicklines
\put(-18,13){\vector(1,1){15}}
\thinlines


\put(5,25){\circle*{1}} \put(20,25){\circle*{1}}
\put(35,25){\circle*{1}} \put(50,25){\circle*{1}}

\put(5,35){\circle*{1}} \put(20,35){\circle*{1}}
\put(35,35){\circle*{1}} \put(50,35){\circle*{1}}

\put(5,25.1){\line(0,1){10}} \put(20,25.1){\line(0,1){10}}
\put(50,25.1){\line(0,1){10}}

\put(5,35.1){\line(3,-2){15}} \put(20,35.1){\line(3,-2){15}}
\put(35,35.1){\line(3,-2){15}} \put(35,35.1){\line(-3,-2){15}}
\put(50,35.1){\line(-3,-2){15}}

\dottedline[$\cdot$]{1}(13.5,40)(13.5,19)
\put(0,21.5){$e_1-e_3$} \put(15,21.5){$e_2-e_4$}
\put(30,21.5){$e_3-e_5$} \put(45,21.5){$e_3+e_5$}

\put(0,37){$e_1-e_4$} \put(15,37){$e_2-e_5$}
\put(30,37){$e_2+e_5$} \put(45,37){$e_3+e_4$}

\thicklines
\put(50,17){\vector(-1,-1){10}}
\thinlines


\put(5,2){\circle*{1}} \put(20,2){\circle*{1}}
\put(35,2){\circle*{1}}

\put(5,12){\circle*{1}}\put(20,12){\circle*{1}}
\put(35,12){\circle*{1}}

\put(5,2.1){\line(0,1){10}}\put(35,2.1){\line(0,1){10}}

\put(5,12.1){\line(3,-2){15}} \put(20,12.1){\line(3,-2){15}} \put(20,12.1){\line(-3,-2){15}} \put(35,12.1){\line(-3,-2){15}}

\end{picture}
\end{center}
\caption{Root poset of $D_5$, $G(\Phi_2,\Phi_3)$ and its central graph}
\label{f1}
\end{figure}
\end{example}

\subsection{Proof of Lemma \ref{eqgra1}}
We only need to consider the cases when the root system $\Phi$ is irreducible. Denote by $\Phi^{(i)}$ a central graph of $G(\Phi_i,\Phi_{i+1})$ and by $A(\Phi^{(i)})$ the matrix associated with $\Phi^{(i)}$. In view of Lemma \ref{eqgra2} and \ref{eqgra3}, it suffices to show that $A(\Phi^{(i)})$ has full rank for $i=1,\ldots,h-1$. Relying heavily on the pictures of the root posets in \cite{R}, we prove this in a case-by-case fashion.

(\rmnum{1}) The central graph $\Phi^{(i)}$ is empty or contains one point or contains two isolated points (including $A_n^{(i)}$, $B_n^{(i)}$, $C_n^{(i)}$, $G_2^{(i)}$ for all possible $i$; $D_n^{(i)}$, $i\geq n-1$ or $i$ is odd; $E_6^{(i)}$, $i\neq 2,3$; $E_7^{(i)}$, $i\neq 2,3,4,8$; $E_8^{(i)}$, $i\neq 2,3,4,5,8,9,14$; $F_4^{(i)}$, $i\neq 3$). There is nothing to prove since the matrix $A(\Phi^{(i)})$ has no column.

(\rmnum{2}) The central graph $\Phi^{(i)}$ is
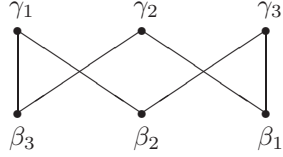
\begin{figure}[H]
\setlength{\unitlength}{1.1mm}
\begin{center}
\begin{picture}(65,18.5)
\put(20,5){\circle*{1}}
\put(35,5){\circle*{1}} \put(50,5){\circle*{1}}

\put(20,15){\circle*{1}}
\put(35,15){\circle*{1}} \put(50,15){\circle*{1}}

\put(20,5.1){\line(0,1){10}}
\put(50,5.1){\line(0,1){10}}

\put(20,15.1){\line(3,-2){15}}
\put(35,15.1){\line(3,-2){15}} \put(35,15.1){\line(-3,-2){15}}
\put(50,15.1){\line(-3,-2){15}}

\put(19,1.2){$\beta_3$} \put(34,1.2){$\beta_2$}
\put(49,1.2){$\beta_1$}

\put(19,17.2){$\gamma_1$}
\put(34,17.2){$\gamma_2$} \put(49,17.2){$\gamma_3$}
\end{picture}
\end{center}
\end{figure}
(including $D_n^{(i)}$ for $i<n-1$ and $i$ is even; $E_6^{(2)}$; $E_7^{(i)}$, i=2,4,8; $E_8^{(i)}$, i=2,4,8,14; $F_4^{(3)}$). The associated matrix is
\begin{equation}\label{amatrix}
A(\Phi^{(i)})=\left(
                    \begin{array}{ccc}
                      0 & N_{\beta_1,-\gamma_2} & N_{\beta_1,-\gamma_3} \\
                      N_{\beta_2,-\gamma_1} & 0 & N_{\beta_2,-\gamma_3} \\
                      N_{\beta_3,-\gamma_1} & N_{\beta_3,-\gamma_2} & 0 \\
                    \end{array}
                  \right)
\end{equation}
Denote $\alpha_1:=\gamma_2-\beta_3$, $\alpha_2:=\gamma_3-\beta_1$ and $\alpha_3:=\gamma_1-\beta_2$. Then $\alpha_1$, $\alpha_2$ and $\alpha_3$ are simple roots. It is evident (by checking the pictures of the root posets of $D_n$, $E_6$, $E_7$, $E_8$ and $F_4$ in \cite{R}) that $\Phi^{(i)}$ is an induced subgraph of a ``cube'' (e.g., see $D_5^{(2)}$ in Figure \ref{f1}) in the root poset of $\Phi$. By symmetry of the cube (see \cite{R}) we also have $\alpha_1=\gamma_3-\beta_2$, $\alpha_2=\gamma_1-\beta_3$ and $\alpha_3=\gamma_2-\beta_1$. Moreover, $\beta_1-\alpha_1=\beta_2-\alpha_2=\beta_3-\alpha_3$ is a positive root, which we denote by $\theta$. The Jacobi identity implies
\[
[[E_{\theta}, E_{\alpha_1}], E_{-\gamma_2}]+[[E_{\alpha_1},E_{-\gamma_2}],E_{\theta}]+[[E_{-\gamma_2},E_{\theta}],E_{\alpha_1}]=0.
\]
It follows that
\begin{equation}\label{3.1.1}
N_{\theta,\alpha_1}N_{\beta_1,-\gamma_2}
+N_{\alpha_1,-\gamma_2}N_{\alpha_1-\gamma_2,\theta}
+N_{-\gamma_2,\theta}N_{-\gamma_2+\theta,\alpha_1}=0.
\end{equation}
If $\Phi^{(i)}\neq F_4^{(3)}$, then all roots in $\Phi$ have the same length. So $N_{\alpha,\beta}=0,\pm1$ for $\alpha,\beta\in\Phi$ by Theorem \ref{Chavelley basis}. Since $N_{\theta,\alpha_1}$, $N_{\beta_1,-\gamma_2}$, $N_{\alpha_1,-\gamma_2}$ and $N_{\alpha_1-\gamma_2,\theta}$ are nonzero by Theorem \ref{Chavelley basis}, we must have $N_{-\gamma_2,\theta}N_{-\gamma_2+\theta,\alpha_1}=0$ by $(\ref{3.1.1})$. We learn from Lemma \ref{Chavelley basis2} that $N_{\alpha_1,-\gamma_2}=-N_{\beta_3,-\gamma_2}$ and $N_{\alpha_1-\gamma_2,\theta}=N_{\theta,\alpha_3}$. This yields
\[
N_{\theta,\alpha_1}N_{\beta_1,-\gamma_2}
=N_{\theta,\alpha_3}N_{\beta_3,-\gamma_2}.
\]
Similarly, one has
\[
N_{\theta,\alpha_2}N_{\beta_2,-\gamma_3}
=N_{\theta,\alpha_1}N_{\beta_1,-\gamma_3}\quad \mbox{and}\quad N_{\theta,\alpha_3}N_{\beta_3,-\gamma_1}
=N_{\theta,\alpha_2}N_{\beta_2,-\gamma_1}.
\]
Therefore
\[
\begin{aligned}
\det (A(\Phi^{(i)}))=&N_{\beta_1,-\gamma_2}N_{\beta_2,-\gamma_3}N_{\beta_3,-\gamma_1}
+N_{\beta_1,-\gamma_3}N_{\beta_2,-\gamma_1}N_{\beta_3,-\gamma_2}\\
=&2N_{\beta_1,-\gamma_2}N_{\beta_2,-\gamma_3}N_{\beta_3,-\gamma_1}\neq0,
\end{aligned}
\]
that is, $A(\Phi^{(i)})$ has full rank. If $\Phi^{(i)}= F_4^{(3)}$, the positive roots of $F_4$ can be written as (\cite{K}, 2.88)
\begin{equation*}
\left\{
\begin{aligned}
&e_i\qquad\qquad i=1,2,3,4\\
&e_i\pm e_j\ \quad 1\leq i<j\leq4\\
&\frac{1}{2}(e_1\pm e_2\pm e_3\pm e_4),
\end{aligned}
\right.
\end{equation*}
where $\{e_1,e_2,e_3,e_4\}$ is an orthonormal basis of $\bbR^4$. The corresponding simple roots are $\frac{1}{2}(e_1-e_2-e_3-e_4)$, $e_4$,
$e_3-e_4$ and $e_2-e_3$ (represented by $1, 2, 3, 4$ respectively in the figure below).

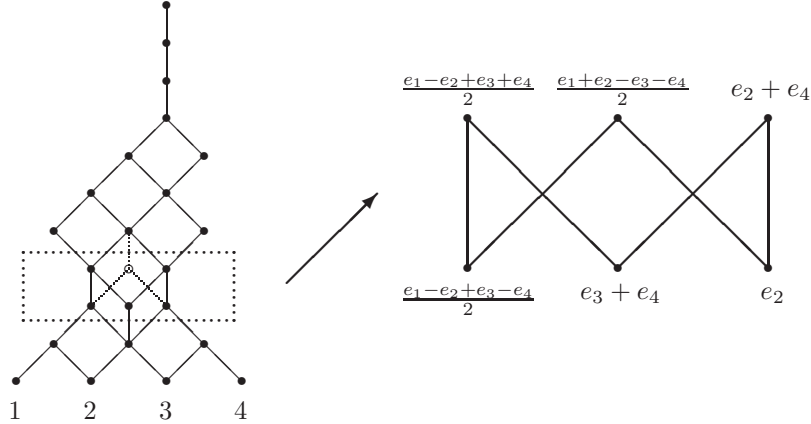
\begin{figure}[H]
\setlength{\unitlength}{1mm}
\begin{center}
\begin{picture}(5,53)
\put(-30,55){\circle*{1}}
\put(-30,50){\circle*{1}}
\put(-30,45){\circle*{1}}
\put(-30,40){\circle*{1}}
\put(-35,35){\circle*{1}}\put(-25,35){\circle*{1}}
\put(-40,30){\circle*{1}}\put(-30,30){\circle*{1}}
\put(-45,25){\circle*{1}} \put(-35,25){\circle*{1}}
\put(-25,25){\circle*{1}}
\put(-40,20){\circle*{1}}\put(-35,20){\circle{1}}
\put(-30,20){\circle*{1}}
\put(-40,15){\circle*{1}}\put(-35,15){\circle*{1}}
\put(-30,15){\circle*{1}}
\put(-45,10){\circle*{1}}\put(-35,10){\circle*{1}}
\put(-25,10){\circle*{1}}
\put(-50,5){\circle*{1}}\put(-40,5){\circle*{1}}
\put(-30,5){\circle*{1}}\put(-20,5){\circle*{1}}

\put(-30,55){\line(0,-1){5}}
\put(-30,50){\line(0,-1){5}}
\put(-30,45){\line(0,-1){5}}
\put(-30,40){\line(-1,-1){5}}\put(-30,40){\line(1,-1){5}}
\put(-35,35){\line(-1,-1){5}}\put(-35,35){\line(1,-1){5}}
\put(-25,35){\line(-1,-1){5}}
\put(-40,30){\line(-1,-1){5}}\put(-40,30){\line(1,-1){5}}
\put(-30,30){\line(-1,-1){5}} \put(-30,30){\line(1,-1){5}}
\put(-45,25){\line(1,-1){5}}
\put(-35,25){\line(-1,-1){5}} \dottedline{0.5}(-35,25)(-35,20)
\put(-35,25){\line(1,-1){5}} \put(-25,25){\line(-1,-1){5}}
\put(-40,20){\line(0,-1){5}}\put(-40,20){\line(1,-1){5}}
\dottedline{0.5}(-40,15)(-35,20)\dottedline{0.5}(-30,15)(-35,20)
\put(-30,20){\line(-1,-1){5}}\put(-30,20){\line(0,-1){5}}
\put(-40,15){\line(-1,-1){5}}\put(-40,15){\line(1,-1){5}}
\put(-35,15){\line(0,-1){5}}
\put(-30,15){\line(1,-1){5}}\put(-30,15){\line(-1,-1){5}}
\put(-45,10){\line(-1,-1){5}}\put(-45,10){\line(1,-1){5}}
\put(-35,10){\line(-1,-1){5}}\put(-35,10){\line(1,-1){5}}
\put(-25,10){\line(-1,-1){5}}\put(-25,10){\line(1,-1){5}}

\put(-51,0){1}\put(-41,0){2}\put(-31,0){3}\put(-21,0){4}
\dottedline[$\cdot$]{1}(-21,22)(-21,13)(-49,13)(-49,22)(-21,22)

\thicklines
\put(-14,18){\vector(1,1){12}}


\put(10,20){\circle*{1}}
\put(30,20){\circle*{1}} \put(50,20){\circle*{1}}

 \put(10,40){\circle*{1}}
\put(30,40){\circle*{1}} \put(50,40){\circle*{1}}

\put(10,40){\line(0,-1){20}}\put(10,40){\line(1,-1){20}}
\put(30,40){\line(-1,-1){20}}\put(30,40){\line(1,-1){20}}
\put(50,40){\line(-1,-1){20}}\put(50,40){\line(0,-1){20}}

\put(1,15){$\frac{e_1-e_2+e_3-e_4}{2}$}
\put(25,16){$e_3+e_4$} \put(49,16){$e_2$}

\put(1,43){$\frac{e_1-e_2+e_3+e_4}{2}$}
\put(21.5,43){$\frac{e_1+e_2-e_3-e_4}{2}$} \put(45,43){$e_2+e_4$}

\end{picture}
\end{center}
\caption{The root poset of $F_4$ and the central graph $F_4^{(3)}$}
\end{figure}

In $F_4^{(3)}$, we have $\beta_1=e_2$, $\beta_2=e_3+e_4$ and $\beta_3=\frac{1}{2}(e_1-e_2+e_3-e_4)$. Moreover, $\gamma_1=\frac{1}{2}(e_1-e_2+e_3+e_4)$, $\gamma_2=\frac{1}{2}(e_1+e_2-e_3-e_4)$ and $\gamma_3=e_2+e_4$. It follows from Theorem \ref{Chavelley basis} and (\ref{amatrix}) that
\[
\det (A(F_4^{(3)}))\equiv\left|
                         \begin{array}{ccc}
                           0 & 1 & 1 \\
                           1 & 0 & 1 \\
                           1 & 2 & 0 \\
                         \end{array}
                       \right|\equiv1\neq0\ (\mathrm{mod} 2).
\]
In other words, the matrix $A(F_4^{(3)})$ has full rank.

(\rmnum{3}) The central graph $\Phi^{(i)}$ is
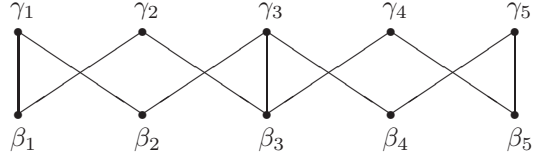
\begin{figure}[H]
\setlength{\unitlength}{1.1mm}
\begin{center}
\begin{picture}(65,18.5)
\put(5,5){\circle*{1}} \put(20,5){\circle*{1}}
\put(35,5){\circle*{1}} \put(50,5){\circle*{1}}
\put(65,5){\circle*{1}}

\put(5,15){\circle*{1}} \put(20,15){\circle*{1}}
\put(35,15){\circle*{1}} \put(50,15){\circle*{1}}
\put(65,15){\circle*{1}}

\put(5,5.1){\line(0,1){10}} \put(35,5.1){\line(0,1){10}}
\put(65,5.1){\line(0,1){10}}

\put(5,15.1){\line(3,-2){15}} \put(20,15.1){\line(3,-2){15}}
\put(35,15.1){\line(3,-2){15}} \put(50,15.1){\line(3,-2){15}}

\put(20,15.1){\line(-3,-2){15}} \put(35,15.1){\line(-3,-2){15}}
\put(50,15.1){\line(-3,-2){15}} \put(65,15.1){\line(-3,-2){15}}

\put(4,1.2){$\beta_1$} \put(19,1.2){$\beta_2$}
\put(34,1.2){$\beta_3$} \put(49,1.2){$\beta_4$}
\put(64,1.2){$\beta_5$}

\put(4,17.2){$\gamma_1$} \put(19,17.2){$\gamma_2$}
\put(34,17.2){$\gamma_3$} \put(49,17.2){$\gamma_4$}
\put(64,17.2){$\gamma_5$}
\end{picture}
\end{center}
\end{figure}
(including $E_6^{(3)}$; $E_7^{(3)}$; $E_8^{(3)}$). Then $N_{\beta_i,-\gamma_j}=0,\pm1$ for $1\leq i, j\leq 5$. We get
\[
\det (A(\Phi^{(i)}))\equiv\left|
                          \begin{array}{ccccc}
                            1 & 1 & 0 & 0 & 0 \\
                            1 & 0 & 1 & 0 & 0 \\
                            0 & 1 & 1 & 1 & 0 \\
                            0 & 0 & 1 & 0 & 1 \\
                            0 & 0 & 0 & 1 & 1 \\
                          \end{array}
                        \right|
\equiv1\neq0\ (\mathrm{mod} 2).
\]

(\rmnum{4}) The central graph $\Phi^{(i)}$ is
\begin{figure}[H]
\setlength{\unitlength}{1.1mm}
\begin{center}
\begin{picture}(65,18.5)
\put(5,5){\circle*{1}} \put(20,5){\circle*{1}}
\put(35,5){\circle*{1}} \put(50,5){\circle*{1}}
\put(65,5){\circle*{1}}

\put(5,15){\circle*{1}} \put(20,15){\circle*{1}}
\put(35,15){\circle*{1}} \put(50,15){\circle*{1}}
\put(65,15){\circle*{1}}

\put(5,5.1){\line(0,1){10}} \put(65,5.1){\line(0,1){10}}

\put(5,15.1){\line(3,-2){15}} \put(20,15.1){\line(3,-2){15}}
\put(35,15.1){\line(3,-2){15}} \put(50,15.1){\line(3,-2){15}}

\put(20,15.1){\line(-3,-2){15}} \put(35,15.1){\line(-3,-2){15}}
\put(50,15.1){\line(-3,-2){15}} \put(50,15.1){\line(-3,-1){30}}
\put(65,15.1){\line(-3,-2){15}}

\put(4,1.2){$\beta_1$} \put(19,1.2){$\beta_2$}
\put(34,1.2){$\beta_3$} \put(49,1.2){$\beta_4$}
\put(64,1.2){$\beta_5$}

\put(4,17.2){$\gamma_1$} \put(19,17.2){$\gamma_2$}
\put(34,17.2){$\gamma_3$} \put(49,17.2){$\gamma_4$}
\put(64,17.2){$\gamma_5$}
\end{picture}
\end{center}
\end{figure}
(including $E_8^{(9)}$). Then
\[
\det A((\Phi^{(i)}))\equiv\left|
                          \begin{array}{ccccc}
                            1 & 1 & 0 & 0 & 0 \\
                            1 & 0 & 1 & 1 & 0 \\
                            0 & 1 & 0 & 1 & 0 \\
                            0 & 0 & 1 & 0 & 1 \\
                            0 & 0 & 0 & 1 & 1 \\
                          \end{array}
                        \right|
\equiv1\neq0\ (\mathrm{mod} 2).
\]

(\rmnum{5}) The central graph $\Phi^{(i)}$ is
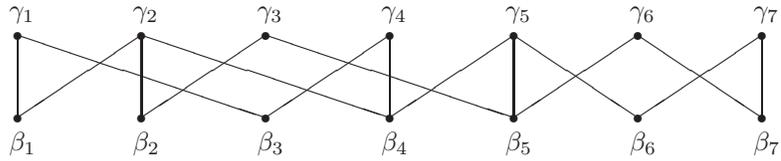
\begin{figure}[H]
\setlength{\unitlength}{1.1mm}
\begin{center}
\begin{picture}(95,18.5)
\put(5,5){\circle*{1}} \put(20,5){\circle*{1}}
\put(35,5){\circle*{1}} \put(50,5){\circle*{1}}
\put(65,5){\circle*{1}} \put(80,5){\circle*{1}}
\put(95,5){\circle*{1}}

\put(5,15){\circle*{1}} \put(20,15){\circle*{1}}
\put(35,15){\circle*{1}} \put(50,15){\circle*{1}}
\put(65,15){\circle*{1}} \put(80,15){\circle*{1}}
\put(95,15){\circle*{1}}

\put(5,5.1){\line(0,1){10}} \put(20,5.1){\line(0,1){10}}
\put(50,5.1){\line(0,1){10}} \put(65,5.1){\line(0,1){10}}
\put(95,5.1){\line(0,1){10}}

\put(5,15.1){\line(3,-1){30}} \put(20,15.1){\line(3,-1){30}}
\put(35,15.1){\line(3,-1){30}} \put(65,15.1){\line(3,-2){15}}
\put(80,15.1){\line(3,-2){15}}

\put(20,15.1){\line(-3,-2){15}} \put(35,15.1){\line(-3,-2){15}}
\put(50,15.1){\line(-3,-2){15}} \put(65,15.1){\line(-3,-2){15}}
\put(80,15.1){\line(-3,-2){15}} \put(95,15.1){\line(-3,-2){15}}

\put(4,1.2){$\beta_1$} \put(19,1.2){$\beta_2$}
\put(34,1.2){$\beta_3$} \put(49,1.2){$\beta_4$}
\put(64,1.2){$\beta_5$} \put(79,1.2){$\beta_6$}
\put(94,1.2){$\beta_7$}

\put(4,17.2){$\gamma_1$} \put(19,17.2){$\gamma_2$}
\put(34,17.2){$\gamma_3$} \put(49,17.2){$\gamma_4$}
\put(64,17.2){$\gamma_5$} \put(79,17.2){$\gamma_6$}
\put(94,17.2){$\gamma_7$}
\end{picture}
\end{center}
\end{figure}
(including $E_8^{(5)}$). Then
\[
\det A((\Phi^{(i)}))\equiv\left|
                          \begin{array}{ccccccc}
                            1 & 1 & 0 & 0 & 0 & 0 & 0 \\
                            0 & 1 & 1 & 0 & 0 & 0 & 0 \\
                            1 & 0 & 0 & 1 & 0 & 0 & 0 \\
                            0 & 1 & 0 & 1 & 1 & 0 & 0 \\
                            0 & 0 & 1 & 0 & 1 & 1 & 0 \\
                            0 & 0 & 0 & 0 & 1 & 0 & 1 \\
                            0 & 0 & 0 & 0 & 0 & 1 & 1 \\
                          \end{array}
                        \right|
\equiv1\neq0\ (\mathrm{mod} 2).
\]
%
%

%
%


\end{document}